\documentclass[11pt, twoside]{amsart}

\title[Frobenius monoidal functors from ambiadjunctions]{Frobenius monoidal functors from ambiadjunctions and their lifts to Drinfeld centers}
\date{April 17, 2025}

\author{Johannes Flake}
\address{Mathematical Institute, University of Bonn, Endenicher Allee 60, 53115 Bonn, Germany}
\email{flake@math.uni-bonn.de}
\urladdr{https://johannesflake.net}

\author{Robert Laugwitz}
\address{School of Mathematical Sciences,
University of Nottingham, University Park, Nottingham, NG7 2RD, UK}
\email{robert.laugwitz@nottingham.ac.uk}
\urladdr{https://www.nottingham.ac.uk/mathematics/people/robert.laugwitz}

\author{Sebastian Posur}
\address{University of Münster,
Fachbereich Mathematik und Informatik,
Einsteinstraße 62,
48149 Münster,
Germany}
\email{sebastian.posur@uni-muenster.de}

\usepackage{import}
\usepackage{amsmath,amsfonts,amsthm,amssymb}
\usepackage[alphabetic,abbrev,nobysame]{amsrefs}

\usepackage[english]{babel}

\usepackage{url}
\usepackage{fancyhdr}
\usepackage{graphicx}
\usepackage{microtype}
\usepackage{wrapfig, caption}
\usepackage[
colorlinks=true,
linkcolor=black, 
anchorcolor=black,
citecolor=black, 
urlcolor=black, 
]{hyperref}
\usepackage{cleveref} 
\usepackage{xcolor}
\usepackage{geometry}
\usepackage[all]{xy}
\usepackage{tikz}
\usepackage{dsfont}
\usetikzlibrary{calc}
\usepackage[shortlabels]{enumitem}
\usetikzlibrary{arrows}

\makeatletter
\newcommand{\superimpose}[2]{%
  {\ooalign{$#1\@firstoftwo#2$\cr\hfil$#1\@secondoftwo#2$\hfil\cr}}}
\makeatother

\newcommand{\leftexpsub}[3]{{\vphantom{#3}}^{#1}_{#2}{#3}}

\newcommand{\lYD}[1]{\leftexpsub{#1}{#1}{\mathbf{YD}}}

\newcommand\longmapsfrom{\mathrel{\reflectbox{\ensuremath{\longmapsto}}}}

\newcommand{\oop}{\mathrm{op}}

\newcommand{\BiMod}[1]{\mathbf{BiMod}\text{-}#1}

\newcommand{\Cat}{\mathbf{Cat}}

\newcommand{\Catlax}{\mathbf{Cat}^{\otimes}_{\mathrm{lax}}}
\newcommand{\Catoplax}{\mathbf{Cat}^{\otimes}_{\mathrm{oplax}}}

\newcommand{\lMod}[1]{#1\text{-}\mathbf{Mod}}

\newcommand{\rMod}[1]{\mathbf{Mod}\text{-}#1}

\newcommand{\projr}[2]{\mathrm{rproj}_{{#1},{#2}}}
\newcommand{\projl}[2]{\mathrm{lproj}_{{#1},{#2}}}
\newcommand{\projrnoarg}{\mathrm{rproj}}
\newcommand{\projlnoarg}{\mathrm{lproj}}

\newcommand{\iprojr}[2]{\mathrm{rproj}^{L,L\dashv G}_{{#1},{#2}}}
\newcommand{\iprojl}[2]{\mathrm{lproj}^{L,L\dashv G}_{{#1},{#2}}}

\newcommand{\llineator}{\mathrm{llinor}}
\newcommand{\rlineator}{\mathrm{rlinor}}

\newcommand{\rev}{\otimes\text{-}\oop}

\newcommand{\Iso}{\operatorname{Iso}}
\newcommand{\Aut}{\operatorname{Aut}}

\newcommand{\coev}{\operatorname{coev}}

\newcommand{\ev}{\operatorname{ev}}

\newcommand{\End}{\operatorname{End}}
\newcommand{\Frob}{\mathbf{Frob}}
\newcommand{\Hom}{\operatorname{Hom}}
\newcommand{\Ind}{\operatorname{Ind}}
\newcommand{\CoInd}{\operatorname{CoInd}}

\newcommand\id{{\operatorname{id}}}

\newcommand{\isomorph}{\stackrel{\sim}{\longrightarrow}}
\newcommand{\lax}{\operatorname{lax}}
\newcommand{\oplax}{\operatorname{oplax}}
\newcommand{\unit}{\operatorname{unit}}
\newcommand{\counit}{\operatorname{counit}}

\newcommand{\Emb}{\operatorname{Emb}}

\newcommand{\one}{\mathds{1}}

\newcommand{\reg}{\mathrm{reg}}

\newcommand{\Res}{\operatorname{Res}}

\newcommand{\tr}{\operatorname{tr}}

\newcommand{\inner}[1]{\left\langle #1\right\rangle}

\newcommand{\Vect}{\mathbf{Vect}}

\newcommand{\sfC}{\mathsf{C}}
\newcommand{\sfG}{\mathsf{G}}
\newcommand{\sfK}{\mathsf{K}}

\providecommand{\fr}[1]{\mathfrak{#1}}

\newcommand{\mC}{\mathbb{C}}

\newcommand{\cC}{\mathcal{C}}
\newcommand{\cD}{\mathcal{D}}
\newcommand{\cB}{\mathcal{B}}

\newcommand{\cM}{\mathcal{M}}
\newcommand{\cN}{\mathcal{N}}

\newcommand{\cZ}{\mathcal{Z}}

\newcommand{\mainfun}{G}
\newcommand{\rightadj}{R}
\newcommand{\leftadj}{L}
\newcommand{\objCa}{A}
\newcommand{\objCb}{B}

\newcommand{\objDx}{X}

\newtheoremstyle{mystyle}
  {0.5cm}                   
  {0.5cm}                   
  {\normalfont}           
  {}                      %                      %
  {\itfont\bfseries} 
  {:}                     
  {0.3cm}              
  {\thmname{#1}}

\newtheoremstyle{defstyle}
  {0.5cm}                   
  {0.5cm}                   
  {\normalfont}           
  {}     
  {\normalfont\bfseries}  
  {:}                     
  {0.3cm}              
  {\thmname{#1}\thmnumber{ #2}\thmnote{ (#3)}}

\newtheorem*{rep@theorem}{\rep@title}
\newcommand{\newreptheorem}[2]{%
\newenvironment{rep#1}[1]{%
 \def\rep@title{#2 \ref{##1}}%
 \begin{rep@theorem}}%
 {\end{rep@theorem}}}
\makeatother

\newtheorem{theorem}{Theorem}[section]

\newtheorem{proposition}[theorem]{Proposition}
\newreptheorem{proposition}{Proposition}
\newtheorem{corollary}[theorem]{Corollary}
\newreptheorem{corollary}{Corollary}
\newtheorem{lemma}[theorem]{Lemma}

\newtheorem*{theorem*}{Theorem}
\newreptheorem{theorem}{Theorem}

\newtheorem{introtheorem}{Theorem}

\newtheorem{introcorollary}[introtheorem]{Corollary}

\theoremstyle{definition}
\newtheorem{definition}[theorem]{Definition}
\newtheorem{notation}[theorem]{Notation}

\newtheorem{context}[theorem]{Context}

\theoremstyle{remark}
\newtheorem{example}[theorem]{Example}
\newtheorem{remark}[theorem]{Remark}

\numberwithin{equation}{section}

\tikzset{mylabel/.style={fill=white,font=\small},font=\small}


\begin{document}

\newgeometry{top=3cm}

\begin{abstract}
We identify general conditions, formulated using the projection formula morphisms, for a functor that is simultaneously left and right adjoint to a strong monoidal functor to be a Frobenius monoidal functor.
    Moreover, we identify stronger conditions for the adjoint functor to extend to a braided Frobenius monoidal functor on Drinfeld centers building on our previous work in [\href{https://arxiv.org/abs/2402.10094}{arXiv:2402.10094}].
    As an application, we  construct concrete examples of (braided) Frobenius monoidal functors obtained from morphisms of Hopf algebras via  induction.  
\end{abstract}

\maketitle

\vspace{-10pt}

\section{Introduction}

\subsection{Motivation and main results}

Consider an \emph{ambiadjunction} $F\dashv G\dashv F$, i.e.\ a pair of functors $G\colon \cC\to \cD$ and $F\colon\cD\to \cC$ together with adjunctions $F\dashv G$ and $G\dashv F$. In this case $F$ (and $G$) is often called a \emph{Frobenius functor}. Assuming that $G$ is a strong monoidal functor, $F$ obtains both a \emph{lax} and an \emph{oplax} monoidal structure \cite{KelDoc} 
$$\lax^F_{X,Y}\colon F(X)\otimes F(Y) \to F(X\otimes Y), \qquad \oplax^F_{X,Y}\colon F(X\otimes Y)\to F(X)\otimes F(Y),$$
which are typically not isomorphisms.
It is natural to ask whether the lax and oplax structures make the Frobenius functor $F$ a \emph{Frobenius monoidal functor}, i.e.\ whether $\lax^F$ and $\oplax^F$ satisfy compatibility conditions similar to that of a Frobenius algebra. 
In general, this is false, and a general connection between Frobenius functors and Frobenius monoidal functors does not appear to be known, cf.\ \cite{nlab:frobenius_monoidal_functor} where the notions are called ``unrelated''. 

\smallskip

In this paper, we identify general conditions for the Frobenius functor $F$ to be a Frobenius monoidal functor based on natural transformations called \emph{projection formula morphisms}. These are left and right projection formula morphisms associated with both adjunctions $G\dashv F$ and $F\dashv G$ (where $A \in \cC$, $X \in \cD$):
\begin{gather}
\objCa \otimes F\objDx \xrightarrow{\projl{\objCa}{\objDx}^{F, G \dashv F}} F( \mainfun\objCa \otimes \objDx ),
\qquad\qquad
F\objDx \otimes \objCa \xrightarrow{\projr{\objDx}{\objCa}^{F, G \dashv F}} F(  \objDx \otimes \mainfun\objCa ),\label{eq:proj-formula-intro1}
\\
F( \mainfun\objCa \otimes \objDx ) \xrightarrow{\projl{\objCa}{\objDx}^{F, F \dashv G}}  \objCa \otimes F\objDx,\qquad\qquad
F(  \objDx \otimes \mainfun\objCa )  \xrightarrow{\projr{\objDx}{\objCa}^{F, F \dashv G}} F\objDx \otimes \objCa.\label{eq:proj-formula-intro2}
\end{gather}
These projection formula morphisms are a known tool in representation theory, algebraic geometry, K-theory, categorical logic and other areas, see e.g.\ \cites{Lawv,FHM} and the references listed in \cite{FLP2}*{Introduction}. We prove the following result.

\begin{introtheorem}[{See \Cref{theorem:frobenius_functors_frobenius_monoidal}}]\label{thm:A}
Assume that $\projrnoarg^{F, G \dashv F}$ and $\projrnoarg^{F, F \dashv G}$ are mutual inverses. Then, $F\colon \cD\to \cC$ with $\lax^F$ and $\oplax^F$ is a Frobenius monoidal functor.
\end{introtheorem}
The conclusion of \Cref{thm:A} also holds if instead the \emph{left} projection formula morphisms $\projlnoarg^{F, G \dashv F}$ and $\projlnoarg^{F, F \dashv G}$ are  mutual inverses. 

\smallskip

Under the stronger condition that \emph{both} the left and right projection formula morphisms are mutual inverses, we obtain braided Frobenius monoidal functors between \emph{Drinfeld centers}.

\begin{introtheorem}[{See \Cref{thm:ZF}}]\label{thm:B}
Assume that $\projrnoarg^{F, G \dashv F}$ and $\projrnoarg^{F, F \dashv G}$, as well as, $\projlnoarg^{F, G \dashv F}$ and $\projlnoarg^{F, F \dashv G}$ are mutual inverses.
Then $F$ extends to a braided Frobenius monoidal functor
$\cZ( F )\colon \cZ( \cD ) \longrightarrow \cZ( \cC )$.
\end{introtheorem}
The lax and oplax monoidal structures for $\cZ(F)$ are inherited from those of the underlying functor $F$ which $\cZ(F)$ extends in the sense that the diagram
$$
\xymatrix@R=15pt{
\cZ(\cD)\ar[d]\ar[rr]^{\cZ(F)}&&\cZ(\cC)\ar[d]\\
\cD\ar[rr]^F&&\cC
}
$$
is a commutative diagram of Frobenius monoidal functors, where the vertical arrows denote the forgetful functors of the Drinfeld center. 

\Cref{thm:B} extends our previous work in \cite{FLP2}*{Theorem B} where it was shown that $\cZ(F)$ extends to Drinfeld centers as a braided lax monoidal functor given a monoidal adjunction $G\dashv F$, and to an oplax monoidal functor given an oplax monoidal adjunction $F\dashv G$, provided that the corresponding projection formula morphisms are invertible. The results in the present paper show that these lax and oplax structures satisfy the compatibility conditions of a Frobenius monoidal functor if, in addition, the projection formula morphisms of the two adjunctions are mutual inverses.  

\subsection{Background and related work}

The study of Frobenius algebras draws motivation from several applications to mathematical physics in topological and conformal field theory. Frobenius algebras internal to monoidal categories describe boundary conditions of conformal field theories \cites{FRS,FFRS} and their extension structure \cite{HKL}. Special classes of commutative Frobenius algebras in braided tensor categories are used to construct modular tensor categories \cites{KO,DS,LW3,HLR}. 

\smallskip

The \emph{Drinfeld center} (also called \emph{monoidal center}) \cites{Maj2,JS} construction assigns a braided monoidal category $\cZ(\cC)$ to a monoidal category $\cC$ and constitutes a key construction in quantum algebra and quantum topology, see e.g.\ \cites{BK,Kas,EGNO,TV}. The Drinfeld center may be seen as a categorical analogue of the usual center of a monoid, where commutativity is no longer a condition but additional data, consisting of a coherent system of \emph{half-braidings} $X\otimes Y\to Y\otimes X$ that witness $X$ commuting with other objects $Y$ in a manner compatible with tensor products (cf.\ \Cref{sec:centers-endo}).

\smallskip

Frobenius functors are a more general class of functors between monoidal categories than strong monoidal functors. In general, a strong monoidal functor $G\colon \cC\to \cD$ does not lift to a strong monoidal functor between the  Drinfeld centers $\cZ(\cC)$ and $\cZ(\cD)$ and braided strong monoidal functors between (non-degenerate or modular) tensor categories rarely exist. However, this papers' methods show how to obtain braided \emph{Frobenius} monoidal functors between the centers. The more general class of braided Frobenius monoidal functors still preserves commutative Frobenius algebras and can therefore be used to construct new examples of such algebra objects from known ones.

\smallskip

The notion of a Frobenius monoidal functor appeared in \cites{Sz1,Sz2,MS,AM,CMZ}, with focus on the separable case and in \cite{DP} in general.  More recently, Frobenius monoidal functors appeared in \cites{Bal1,Bal,FHL,HLR,Yad}. While \cites{FHL,HLR} construct braided Frobenius monoidal functors between Drinfeld centers of (twisted) $G$-graded vector spaces, the papers \cites{Bal,Yad} are concerned with the question as to when the right adjoint of a strong monoidal functor of abelian tensor categories is a Frobenius monoidal functor. Our present paper contributes new general results to both lines of inquiry. 

\smallskip

The invertibility of the projection formula morphisms plays a key role in the theory of Hopf monads, where they are called (co)Hopf operators \cite{BLV}. Our \Cref{thm:A} is related to results of \cite{Bal} and \cite{Yad} where it was shown, in the case when the projection formula morphisms are invertible and certain (co)monadicity conditions hold, that $F$ is Frobenius monoidal if and only if $F(\one_\cD)$ is a Frobenius algebra. Very recently, \cite{SY}*{Proposition~6.8} showed that certain classes of tensor functors of unimodular finite tensor categories have right adjoints that are Frobenius monoidal. Our result \Cref{thm:A} does not impose such monadicity, finiteness, or unimodularity assumptions.

\subsection{Description of results}

We start with the most general result: Consider a diagram of functors
\[
\begin{tikzpicture}[mylabel/.style={fill=white}]
      \coordinate (r) at (3,0);
      \node (A) {$\cC$};
      \node (B) at ($(A) + (r)$) {$\cD$};
      \draw[->, out = 20, in = 180-20] (A) to node[mylabel]{$G$} (B);
      \draw[<-, out = -20, in = 180+20] (A) to node[mylabel]{$F$} (B);
\end{tikzpicture} 
\]
and the diagram of functors between endofunctor categories
\[
\begin{tikzpicture}[mylabel/.style={fill=white}]
      \coordinate (r) at (5,0);
      \node (A) {$\End(\cC)$};
      \node (B) at ($(A) + (r)$) {$\End(\cD)$};
      \draw[->, out = 15, in = 180-15] (A) to node[mylabel]{$\Gamma := \Hom( F, G )$} (B);
      \draw[<-, out = -15, in = 180+15] (A) to node[mylabel]{$\Phi := \Hom( G, F )$} (B);
\end{tikzpicture} 
\]
obtained by composition with $G$ and $F$. That is,
\[
\Gamma( A ) = \Hom(F,G)(A) = GAF, \qquad 
\Phi( X ) = \Hom(G,F)(X) = FXG,
\]
for $A \in \End( \cC )$, $X \in \End( \cD )$. 
Endofunctor categories are \emph{monoidal categories}, with the tensor product given by composition, even when $\cC$ and $\cD$ are not themselves monoidal categories. We observe that 
\begin{itemize}
    \item an adjunction $G \dashv F$ induces an oplax-lax adjunction $\Gamma \dashv \Phi$,
    \item an adjunction $F \dashv G$ induces an oplax-lax adjunction $\Phi \dashv \Gamma$.
\end{itemize}
Our first result shows that \emph{Frobenius functors} induce \emph{Frobenius monoidal functors} on endofunctor categories.

\begin{introtheorem}[{See \Cref{theorem:relation_frobenius_monoidal_and_ambi}}]\label{thm:C}
If $G$ is a Frobenius functor, i.e.\ we have an ambiadjunction $F \dashv G \dashv F$, then $\Phi$ and $\Gamma$ are Frobenius monoidal functors.
\end{introtheorem}

In fact, the above theorem holds for internal (ambi)adjunctions in any bicategory instead of that of categories, functors, and natural transformations. As a special case, for a monoidal category $\cC$ viewed as a single-object bicategory $\mathbf{B}(\cC)$, the above result recovers the fact that if an object $A$ has a left and right dual $B$, then both $A\otimes B$ and $B\otimes A$ are Frobenius algebras in $\cC$, see \Cref{ex-classical-Frob}.

\smallskip

We are particularly interested in the bicategories $\rMod{\cC}$ and $\BiMod{\cC}$ or right $\cC$-modules, respectively, $\cC$-bimodules of a monoidal category $\cC$. Here, the endofunctor categories can be identified as the following monoidal categories:
\begin{gather*}
    \End_{\rMod{\cC}}(\cC)\simeq \cC, \qquad \End_{\BiMod{\cC}}(\cC)\simeq \cZ(\cC) .
\end{gather*}
The monoidal category $\cZ(\cC)$ is the \emph{Drinfeld center} of $\cC$ which is, in addition, braided monoidal. A key idea of the present paper is to apply \Cref{thm:C} to induce a Frobenius monoidal structure for the functor $F$ and for an induced functor on Drinfeld centers.  
For this, we require internal ambiadjunctions in $\rMod{\cC}$, or $\BiMod{\cC}$ starting with a strong monoidal functor $G\colon \cC\to \cD$ and an ambiadjunction 
$F\dashv G\dashv F$.
In this setup (called \Cref{context:main} in the text), we fix lax and oplax monoidal structures on $F$ by adjunction.  
Building on our work in \cite{FLP2}, we employ the projection formula morphisms \eqref{eq:proj-formula-intro1}--\eqref{eq:proj-formula-intro2} as a key tool in our constructions.

\Cref{thm:A} is proved by first showing that the right projection formula morphisms being mutual inverses upgrades the ambiadjunction $F\dashv G\dashv F$ to an ambiadjunction internal to $\rMod{\cC}$, see \Cref{theorem:strong_monoidal_to_ambiadjunction_right}. Under the equivalence $\End_{\rMod{\cC}}(\cC)\simeq \cC$, pre-composing $\Hom(G,F)$ with the strong monoidal functor $\cD\hookrightarrow \End_{\rMod{\cC}}(\cD)$ recovers $F$. However, now, by \Cref{thm:C}, $F$ is a composition of Frobenius monoidal functors and hence itself Frobenius monoidal.

\smallskip

We remark that the datum of an ambiadjunction $F\dashv G\dashv F$ is not canonical. There exist non-equivalent data of such ambiadjunctions and the category formed by them is equivalent to $\Aut(F)$, see \Cref{remark:ambiadjunctions_classified_by_aut}. Hence, whether the results \Cref{thm:A} or \Cref{thm:B} hold is dependent on the choice of an ambiadjunction. We can, however, show that if there exists an ambiadjunction $F\dashv G\dashv F$ such that the right projection formula morphisms are mutually inverse, then the space of such ambiadjunction data is governed by $\Aut_\cD(\one)$, see \Cref{proposition:EndF}, where $\one$ denotes the tensor unit.

\smallskip

To prove \Cref{thm:B}, we employ the following characterization of $F$ being a $\cC$-bimodule functor, using that a strong monoidal functor $G\colon\cC\to\cD$ as above induces a $\cC$-bimodule structure on $\cD$ which we denote by $\cD^G$, see \Cref{def:DG}.

\begin{introtheorem}[{See \Cref{theorem:strong_monoidal_to_ambiadjunction}}]\label{thm:D}
The following statements are equivalent.
\begin{enumerate}
    \item The projection formula morphisms \eqref{eq:proj-formula-intro1}--\eqref{eq:proj-formula-intro2} are mutual inverses.
    \item The functor $F\colon \cD^G \rightarrow \cC$ has the structure of a $\cC$-bimodule functor making $F \dashv G \dashv F$ an ambiadjunction of $\cC$-bimodules.
\end{enumerate}
\end{introtheorem}
We can now apply \Cref{thm:C} to the ambiadjunction of $\cC$-bimodules from \Cref{thm:D}. Using strong monoidal functors 
$$\cZ(\cD)\hookrightarrow\End_{\BiMod{\cC}}(\cD^G)\quad \text{ and }\quad \End_{\BiMod{\cC}}(\cC)\simeq \cZ(\cC),$$ the functor $\Hom(G,F)$ yields the Frobenius monoidal functor of \Cref{thm:B}.

\smallskip

The functor $\cZ(F)$ in \Cref{thm:B} is a special case of both the lax and oplax monoidal functors 
$$\cZ(R)\colon \cZ(\cD)\to \cZ(\cC)\quad \text{ and }\quad \cZ(L)\colon \cZ(\cD)\to \cZ(\cC)$$
obtained from a right adjoint $R$ and a left adjoint $L$ to $G$ in \cite{FLP2}*{Theorems~4.10 and 4.15}. For an ambiadjunction $F=R=L$ the induced functors on the Drinfeld centers are equal if the projection formulas are mutual inverses. In this case, we show that no further assumptions are needed so that $F$ is Frobenius monoidal.

In the special case of $G$ being restriction along an inclusion of finite groups, the Frobenius monoidal functor $\cZ(F)$ was already known \cite{FHL}*{Appendix~B} and extended to the twisted case in \cite{HLR}.

\smallskip

The techniques developed in this paper enable the construction of a vast supply of examples of Frobenius monoidal functors obtained, e.g., from induction along morphisms of Hopf algebras $\varphi\colon K\to H$. We  explore this class of examples in \Cref{sec:Hopf}. The datum of an isomorphism $\Ind_\varphi\cong \CoInd_\varphi$ corresponds to a \emph{Frobenius morphism} $\tr\colon H\to K$ (see \Cref{thm:ind-coind}, which is obtained from \cite{MN}). We show that the right (or left) projection formula morphisms are mutual inverses if and only if $\tr$ is a morphism of right (or left, respectively) $H$-comodules (see \Cref{prop:Hopf-right} and \Cref{thm:Hopf}). 

\begin{introcorollary}[{See \Cref{prop:Hopf-right}, \Cref{cor:Ind-Frob}, and \Cref{cor:ZInd-Frob}}]\label{cor:E}
     Assume that $\varphi\colon K\to H$ is a morphism of Hopf algebras such that $H$ is finitely generated projective as a left $K$-module and that there exists a Frobenius morphism $\tr\colon H\to K$. 
\begin{enumerate}[(i)]
    \item
     If $\tr$ is a morphism of right $H$-comodules, then $\Ind_\varphi\colon \lMod{K}\to \lMod{H}$ is a Frobenius monoidal functor.
    \item
If $\tr$ is a morphism of $H$-$H$-bicomodules, then $\Ind_\varphi$ induces a braided Frobenius monoidal functor $\cZ(\Ind_\varphi)\colon \cZ(\lMod{K})\to \cZ(\lMod{H})$.
\item A Frobenius morphism $\tr$ satisfying (i) or (ii) is unique up to multiplication by a non-zero scalar. 
\end{enumerate}
The above functors restrict to the subcategories of finite-dimensional modules and their Drinfeld centers.
\end{introcorollary}

\Cref{cor:E} generalizes the classical result that any finite-dimensional Hopf algebra $H$ can be equipped with the structure of a Frobenius algebra (see \cite{LS2}) which we recover in the following way. If $\cC=\lMod{H}$ and $G\colon \cC\to \Vect$ is the forgetful functor, then isomorphisms $\Ind\cong \CoInd$ such that the right projection formula morphisms are invertible are in bijection with right integrals for $H^*$ (see \Cref{ex:classicalH}). Hence, \Cref{cor:E}(i) tells us that $\Ind(\Bbbk)$ is a Frobenius algebra (in $\lMod{H}$) that can be identified with $H$ as a coalgebra. By \Cref{cor:E}(iii), the space of right integrals for $H^*$ is one-dimensional. Moreover, by \Cref{cor:E}(ii), if $H^*$ is unimodular, then $H$ is a commutative Frobenius algebra in $\cZ(\lMod{H})$, see \Cref{ex:classicalH} and \Cref{cor:Hunimodular}.

\smallskip

The following explicit examples are noteworthy as they demonstrate the difference between and limits of the conditions required in \Cref{thm:A} compared to \Cref{thm:B}:
\begin{itemize}
\item For the \emph{Taft algebra} $T_\ell(\epsilon)$ \cite{Taf} associated to a primitive odd $\ell$-th root of unity $\epsilon$ and the inclusion $\varphi\colon \Bbbk\sfC_\ell\hookrightarrow T_\ell(\epsilon)$ of the group algebra of the cyclic group $\sfC_\ell$, we see that, even though the projection formula morphisms are invertible, the functors $\CoInd_\varphi$ and $\Ind_\varphi$ are \emph{not} isomorphic. Thus, none of the results apply to this example (see \Cref{ex:Taft}). However, $\cZ(\Ind_\varphi)$ and $\cZ(\CoInd_\varphi)$ still exists as lax, respectively, oplax monoidal functors on Drinfeld centers by \cite{FLP2}.
\item The \emph{small quantum group} $u_\epsilon(\fr{sl}_2)$ also contains the group algebra $\Bbbk\sfC_\ell$ as a Hopf subalgebra. In this case, the conditions required for \Cref{thm:A} hold so that the associated induction functor becomes a Frobenius monoidal functor. However, the stronger conditions required for \Cref{thm:B} are false, see \Cref{ex:uqsl2} and \Cref{ex:uqsl2-cont}.
\end{itemize}
A more in-depth analysis, using integrals and unimodularity conditions, of the Hopf algebra case will appear in a follow-up paper. There, we shall present more general classes of extensions of Hopf algebras which induce braided Frobenius monoidal functors on their Drinfeld centers. 

\subsection{Summary of content}
We start by providing the necessary preliminaries on Frobenius monoidal functors and ambiadjunctions in \Cref{sec:prelim}. The general relationship \Cref{thm:C} between Frobenius functors and Frobenius monoidal functors on endofunctor categories is proved in \Cref{sec:FrobEndo} in the context of general strict bicategories. The projection formula morphisms are reviewed in \Cref{sec:context-proj} where the main context of an ambiadjunction $F\dashv G\dashv F$, with $G$ strong monoidal, is fixed. This context is needed in the following Sections \ref{section:monoidal_induce_ambi_of_bim}--\ref{sec:F-Frob-mon} and Section \ref{section:functors_on_centers} that establish when $F$ is Frobenius monoidal (proving \Cref{thm:A}), respectively, induces a Frobenius monoidal functor on centers (proving \Cref{thm:B}). As a thorough proof of concept, the final \Cref{sec:Hopf} investigates the results of the paper for categories of modules over Hopf algebras and their centers, described using the model of Yetter--Drinfeld modules. In this section, most examples are found. 

\subsection*{Acknowledgements}
We thank Kenichi Shimizu and Harshit Yadav for interesting comments related to this work. J.~F.~thanks the Max Planck Institute for Mathematics Bonn for the excellent conditions provided while some of this article was written. We also thank the anonymous referees for their helpful comments.

\setcounter{tocdepth}{3}
\tableofcontents

\section{Preliminaries}
\label{sec:prelim}

\subsection{Frobenius monoidal functors}

We recall here the definition of a Frobenius monoidal functor from \cite{AM}*{Section~3.5}, \cite{DP} and establish first properties, including that such functors preserve Frobenius algebras.

Let $\cC$, $\cD$ be monoidal categories.
A \emph{Frobenius monoidal functor} $F\colon \cD\to \cC$ is a functor equipped with a lax monoidal structure $(\lax,\lax_0)$, and an oplax monoidal structure $(\oplax, \oplax_0)$, 
\begin{gather}
\begin{split}
    \lax_{X,Y}\colon F(X)\otimes F(Y)\longrightarrow F(X\otimes Y), \qquad \lax_0\colon \one \longrightarrow F(\one),\\ 
    \oplax_{X,Y}\colon F(X\otimes Y)\longrightarrow F(X)\otimes F(Y), \qquad \oplax_0\colon F(\one)\longrightarrow \one,
\end{split}
\end{gather}
for any objects $X,Y$ of $\cD$,
satisfying the compatibility conditions that the diagrams
\begin{gather}\label{frobmon1}
\vcenter{\hbox{\xymatrix{
&F(X)\otimes F(Y)\otimes F(Z)\ar@/^1pc/[rd]^-{\;\lax_{X,Y}\otimes \id_{F(Z)}}&\\
F(X)\otimes F(Y\otimes Z)\ar@/^1pc/[ru]^-{\id_{F(X)}\otimes \oplax_{Y,Z}\quad}\ar@/_1pc/[rd]_{\lax_{X,Y\otimes Z}} && F(X\otimes Y)\otimes F(Z),\\
&F(X\otimes Y\otimes Z)\ar@/_1pc/[ru]_{\oplax_{X\otimes Y,Z}}&
}}}
\end{gather}
\begin{gather}
\vcenter{\hbox{\xymatrix{
&F(X)\otimes F(Y)\otimes F(Z)\ar@/^1pc/[rd]^-{\id_{F(X)}\otimes \lax_{Y,Z}}&\\
F(X\otimes Y)\otimes F(Z)\ar@/^1pc/[ru]^-{\oplax_{X,Y}\otimes \id_{F(Z)}\hspace{2em}} \ar@/_1pc/[rd]_-{\lax_{X\otimes Y,Z}} && F(X)\otimes F(Y\otimes Z)\\
&F(X\otimes Y\otimes Z)\ar@/_1pc/[ru]_-{\oplax_{X,Y\otimes Z}}&
}}}\label{frobmon2}
\end{gather}
commute for any objects $X,Y,Z$ of $\cD$. We note that any strong monoidal functor is, in particular, Frobenius monoidal \cite{DP}*{Proposition~3}.

Recall that a \emph{Frobenius algebra} in a monoidal category $\cD$ is an object $X$ in $\cD$ which is both an algebra $(X,m\colon X\otimes X\to X,u\colon \one \to X)$ and a coalgebra $(X,\Delta\colon X\to X\otimes X, \varepsilon \colon X\to \one)$ satisfying the compatibility conditions that all diagrams in
\begin{gather}\label{eq:Frobaxiom}
    \vcenter{\hbox{\xymatrix{
&&X\otimes X\otimes X\ar@/^1pc/[drr]^-{m\otimes \id_X}&&\\
X\otimes X\ar@/^1pc/[urr]^-{\id_X\otimes \Delta}\ar@/_1pc/[drr]_-{\Delta\otimes \id_X}\ar[rr]^m&& X\ar[rr]^\Delta &&X\otimes X\\
&&X\otimes X\otimes X\ar@/_1pc/[urr]_-{\id_X\otimes m}&&
    }}}
\end{gather}
commute. A \emph{morphism of Frobenius algebras} is required to be compatible with both the algebra and coalgebra structure of $X$. We denote the category of Frobenius algebras in $\cD$ by $\Frob(\cD)$. It is well-known that morphisms between Frobenius algebras are isomorphisms, see e.g.\ \cite{GV}*{Proposition~2.2}, so $\Frob(\cD)$ is a groupoid.

\begin{lemma}\label{lem:Frobalg-pres}
A Frobenius monoidal functor $F\colon \cD\to \cC$ induces a functor
$\Frob(\cD)\to\Frob(\cC)$, $X\mapsto F(X)$.
\end{lemma}
\begin{proof}Given a Frobenius algebra $X$ in $\cC$, $F(X)$ is both an algebra and a coalgebra using 
\begin{gather*}
    m_{F(X)}=F(m)\lax_{X,X}, \qquad u_{F(X)}=F(u)\lax_{0},\\\Delta_{F(X)}=\oplax_{X,X}F(\Delta), \qquad \counit_{F(X)}=\oplax_{0} F(\varepsilon).
\end{gather*}
We compute the first compatibility condition
\begin{align*}
\Delta_{F(X)}m_{F(X)}&=\oplax_{X,X}F(\Delta)F(m)\lax_{X,X}\\
&=\oplax_{X,X}F(m\otimes X)F(X\otimes\Delta)\lax_{X,X}\\
&=(F(m)\otimes F(X))\oplax_{X\otimes X,X}\lax_{X, X\otimes X}(F(X)\otimes F(\Delta))\\
&=(F(m)\otimes F(X))(\lax_{X,X}\otimes F(X))(F(X)\otimes \oplax_{X, X})(F(X)\otimes F(\Delta))\\
&=(m_{F(X)}\otimes F(X))(F(X)\otimes \Delta_{F(X)}).
\end{align*}
Here, we use that $F$ is a Frobenius functor in the second last equality.
The other compatibility condition from \Cref{eq:Frobaxiom} follows analogously. By naturality of $\lax$ and $\oplax$ it follows that the assignment extends to morphisms of Frobenius algebras.
\end{proof}

Recall that a Frobenius algebra in a braided monoidal category $\cD$ with braiding $\Psi$ is \emph{commutative} if $m\Psi=m$ and \emph{cocommutative} if  $\Psi\Delta=\Delta$. 

\begin{definition}\label{def:braiding-lax-oplax}
Let $\cC,\cD$ be braided monoidal categories.
We say that a lax (or oplax) monoidal functor $F\colon \cD\to \cC$ is \emph{lax braided} (or \emph{oplax braided}, respectively) if the diagram \Cref{eq:Frob-braided-lax} below (or \Cref{eq:Frob-braided-oplax} below, respectively) commutes for all objects $X,Y\in \cD$.
\begin{align}\label{eq:Frob-braided-lax}
    &\vcenter{\hbox{\xymatrix{
    F(X)\otimes F(Y)\ar[rr]^{\Psi^\cC_{F(X),F(Y)}}\ar[d]_{\lax_{X,Y}}&& F(Y)\otimes F(X)\ar[d]_{\lax_{Y,X}}\\
    F(X\otimes Y)\ar[rr]^{F(\Psi^\cD_{X,Y})}&& F(Y\otimes X)
    }}}\\ \label{eq:Frob-braided-oplax}
    &\vcenter{\hbox{\xymatrix{
    F(X\otimes Y)\ar[rr]^{F(\Psi^\cD_{X,Y})}\ar[d]_{\oplax_{X,Y}}&& F(Y\otimes X)\ar[d]_{\oplax_{Y,X}}\\
    F(X)\otimes F(Y)\ar[rr]^{\Psi^\cC_{F(X),F(Y)}}&& F(Y)\otimes F(X)
    }}}
\end{align}
A Frobenius monoidal functor between braided monoidal categories is called \emph{braided} if it is both lax braided and oplax braided.  
\end{definition}

\begin{lemma}
If $F$ is a lax braided (or oplax braided)  functor, then $F(A)$ is a commutative algebra (cocommutative coalgebra) for any commutative algebra (cocommutative coalgebra) $A$ in $\cC$.
\end{lemma}
\begin{proof}
    This is a straightforward computation.
\end{proof}

For lax monoidal functors $F: \cD \rightarrow \cC$ and $E: \cC \rightarrow \cB$, the composite $E \circ F: \cD \rightarrow \cB$ is again a lax monoidal functor with lax structure given by
\[
EF(X) \otimes EF(Y) \xrightarrow{\lax^E_{FX,FY}} E( FX \otimes FY ) \xrightarrow{E( \lax^F_{X,Y})} EF( X \otimes Y )
\]
for $X,Y \in \cD$ and
\[
\one \xrightarrow{\lax^E_0} E(\one) \xrightarrow{E( \lax^F_0)} EF( \one ).
\]
Analogously, the composite of two oplax monoidal functors can be equipped with an oplax monoidal structure.

\begin{lemma}\label{lemma:composition_of_Frobenius_monoidal_functors}
Let $F: \cD \rightarrow \cC$, $E: \cC \rightarrow \cB$ be Frobenius monoidal functors. Then $E \circ F$ equipped with its composite lax and oplax monoidal structures is a Frobenius monoidal functor.
\end{lemma}
\begin{proof}
See \cite{DP}*{Proposition 4}.
\end{proof}

A natural transformation $\eta\colon F\to F'$ between lax monoidal functors $F,F'\colon \cD\to \cC$ is called \emph{monoidal} if the equations
\[
\lax^{F'}_{X,Y} \circ (\eta_X \otimes \eta_Y) = \eta_{X \otimes Y} \circ \lax^{F}_{X,Y}, \qquad \qquad \eta_{\one} \circ \lax^F_0 = \lax^{F'}_0
\]
hold for all $X,Y \in \cD$. Analogously, we have the notion of an \emph{opmonoidal} natural transformation between oplax monoidal functors.

\begin{lemma}\label{lemma:transport_Frobenius_monoidal_structure}
Let $F, F': \cD \rightarrow \cC$ be functors with both a lax and an oplax monoidal structure.
Let $\tau: F \rightarrow F'$ be a natural isomorphism which is both monoidal and opmonoidal.
Then $F$ is a Frobenius monoidal functor if and only if $F'$ is a Frobenius monoidal functor.
\end{lemma}
\begin{proof}
One easily checks that the defining equations of a Frobenius monoidal functor are correctly transported via $\tau$.
\end{proof}

\begin{lemma}\label{lemma:Frobenius_along_equivalence}
Let $F: \cD \rightarrow \cC$ be a lax and oplax monoidal functor. Let $E: \cC \rightarrow \cB$ be a monoidal equivalence. Then $F$ is a Frobenius monoidal functor if and only if $E \circ F$ is a Frobenius monoidal functor.
\end{lemma}
\begin{proof}
Let $E \circ F$ be Frobenius monoidal. Since $E$ is a monoidal equivalence, we have a monoidal functor $E'$ and a monoidal isomorphism $\eta: \id \rightarrow E'E$.
We note that $E'EF$ is Frobenius monoidal by \Cref{lemma:composition_of_Frobenius_monoidal_functors}.
Thus, we can set $\tau := F \xrightarrow{ \id_F \ast \eta} E'EF$. Note that this horizontal composition of natural isomorphisms is again monoidal and opmonoidal. Now, we use \Cref{lemma:transport_Frobenius_monoidal_structure} to deduce that $F$ is indeed Frobenius monoidal.
The other direction follows from \Cref{lemma:composition_of_Frobenius_monoidal_functors}, since $E$ is in particular a Frobenius monoidal functor.
\end{proof}

\subsection{Ambiadjunctions and Frobenius functors}

In this subsection, we use the language of strict bicategories, see \cite{JY21} for a reference or \cite{FLP2}*{Appendix} for further details. We denote horizontal composition of $2$-morphisms in a strict bicategory by $\ast$.

\begin{definition}\label{def:adjunction}
An \emph{(internal) adjunction} in a strict bicategory consists of:
\begin{enumerate}
    \item Objects $\cC, \cD$ and $1$-morphisms $\cC \xrightarrow{G} \cD \xrightarrow{F} \cC$,
    \item $2$-morphisms $\id_{\cC} \xrightarrow{\unit} FG$ and $GF \xrightarrow{\counit} \id_{\cD}$.
\end{enumerate}
These data satisfy the \emph{zigzag identities}:
\[
(\id_F \ast \counit) \circ (\unit \ast \id_F) = \id_F \qquad \qquad (\counit \ast \id_G) \circ (\id_G \ast \unit) = \id_G
\]
The $1$-morphism $G$ is called the \emph{left adjoint}, $F$ the \emph{right adjoint}.
We refer to such an adjunction by $G \dashv F$.
\end{definition}

\begin{lemma}[Uniqueness of adjoints in strict bicategories]\label{lemma:uniqueness_of_adjoints}
Let $\cC \xrightarrow{G} \cD$ be a morphism in a strict bicategory.
Assume we have two right adjoints of $G$:
\begin{itemize}
    \item $\cD \xrightarrow{F} \cC$, $\id_{\cC} \xrightarrow{\unit} FG$ and $GF \xrightarrow{\counit} \id_{\cD}$,
    \item $\cD \xrightarrow{F'} \cC$, $\id_{\cC} \xrightarrow{\unit'} F'G$ and $GF' \xrightarrow{\counit'} \id_{\cD}$.
\end{itemize}
Then there exists a unique $2$-morphism
\[
\theta: F' \rightarrow F
\]
which commutes with the units and counits:
\[
(\theta \ast G) \circ \unit' = \unit
\]
and
\[
\counit \circ (G \ast \theta) = \counit'.
\]
Moreover, $\theta$ is always a $2$-isomorphism  (which we call \emph{comparison isomorphism}).
An analogous statement holds for left adjoints.
\end{lemma}
\begin{proof}
See, e.g., \cite{RiehlContext}*{Proposition 4.4.1} and \cite{RiehlContext}*{Remark 4.4.7}.
\end{proof}

\begin{remark}\label{remark:category_of_adjunctions}
If we fix $G: \cC \rightarrow \cD$, then we get a category whose objects are adjunctions with $G$ as a left adjoint, and a morphism from $(G \dashv F)$ to $(G \dashv F')$ is given by a $2$-morphism $F \xrightarrow{\theta} F'$ which is compatible with the unit and counit in the sense of \Cref{lemma:uniqueness_of_adjoints}.
It directly follows from \Cref{lemma:uniqueness_of_adjoints} that this category is equivalent to a singleton (a category with one object and one morphism). 
In this sense, \emph{$G$ being a left adjoint} is a property in the categorical sense.
An analogous statement holds for a fixed right adjoint.
\end{remark}

\begin{remark}\label{remark:pullback_adjunctions}
A converse of \Cref{lemma:uniqueness_of_adjoints} holds:
assume that we have an isomorphism $\theta: F' \isomorph F$ and an adjunction $G \dashv F$. Then we get another adjunction $G \dashv F'$ by setting $\unit' := (\theta^{-1} \ast G) \circ \unit$ and $\counit' := \counit \circ (G \ast \theta)$. We refer to this adjunction by $\theta( G \dashv F )$. Note that $(G \dashv F)$ and $\theta( G \dashv F )$ are isomorphic objects in the category described in \Cref{remark:category_of_adjunctions}.
\end{remark}

\begin{definition}\label{def:ambiadj}
An \emph{(internal) ambiadjunction} in a strict bicategory consists of the following data:
\begin{enumerate}
    \item Objects $\cC, \cD$ and $1$-morphisms $\cC \xrightarrow{G} \cD \xrightarrow{F} \cC$,
    \item $2$-morphisms
        \begin{align*}
         \unit^{F\dashv G}\colon &\id_{\cD}\to GF, &  \counit^{F\dashv G}\colon& FG\to \id_{\cC}
        \end{align*}
        which make $F$ a left adjoint to $G$,
    \item $2$-morphisms
    \begin{align*}
         \unit^{G\dashv F}\colon &\id_{\cC}\to FG, &  \counit^{G\dashv F}\colon& GF\to \id_{\cD},
        \end{align*}
        which make $F$ a right adjoint to $G$.
\end{enumerate}
We also call these data an \emph{ambiadjunction for $G$} and refer to it by $F \dashv G \dashv F$.
\end{definition}

\begin{remark}
    The situation in \Cref{def:ambiadj} goes by different names in the literature. For instance, in \cite{Mor} $F$ and $G$ are called a \emph{strongly adjoint pair} and in other sources $G$ is called a \emph{Frobenius functor} \cites{CMZ,ShiU}. More generally, if the left and right adjoint of $G$ are just isomorphic, the term \emph{ambidextrous adjunction} is frequently used and shortened to the term ambiadjunction used here. Note that if left and right adjoint are isomorphic, we can define the structure of a left adjoint on the right adjoint and vice versa (see \Cref{remark:ambiadjunction_alt_def}).
\end{remark}

\begin{remark}\label{remark:ambiadjunctions_classified_by_aut}
If we fix $G: \cC \rightarrow \cD$, then we get a category whose objects are ambiadjunctions for $G$, and a morphism from $(F \dashv G \dashv F)$ to $(F' \dashv G \dashv F')$ is given by a $2$-morphism $F \xrightarrow{\tau} F'$ which is compatible with the units and counits of both adjunctions in the sense of \Cref{lemma:uniqueness_of_adjoints}.
It directly follows from \Cref{lemma:uniqueness_of_adjoints} that this category is a setoid, i.e., all morphisms are isomorphisms and there is at most one morphism between any pair of objects.

If $(F \dashv G \dashv F)$ is any object in that category, then any $\theta \in \Aut( F )$ defines another object in that category whose first adjunction is still given by $F \dashv G$, and whose second adjunction is given by $\theta(G \dashv F)$ as it was described in \Cref{remark:pullback_adjunctions}. It is easy to verify that this assignment defines an equivalence of setoids from $\Aut(F)$ to the category of ambiadjunctions for $G$. In particular, we see that \emph{$G$ being part of an ambiadjunction} is not a property in the categorical sense, but an actual datum.
\end{remark}

\begin{remark}\label{remark:ambiadjunction_alt_def}
Another way to define an (internal) ambiadjunction is given by specifying the following data:
\begin{enumerate}
    \item Objects $\cC, \cD$ and a $1$-morphism $\cC \xrightarrow{G} \cD$. 
    \item Adjunctions $G \dashv R$ and $L \dashv G$.
    \item A $2$-isomorphism $\iota: R \isomorph L$.
\end{enumerate}
For a fixed $G$, we again get a category of such data, where a morphism from $(L \dashv G \dashv R, R \xrightarrow{\iota} L)$ to $(L' \dashv G \dashv R', R' \xrightarrow{\iota'} L')$ is given by a pair of morphisms $\theta: R \rightarrow R'$ and $\zeta: L \rightarrow L'$ which are compatible with the units and counits in the sense of \Cref{lemma:uniqueness_of_adjoints}, and for which we have $\iota' \circ \theta = \zeta \circ \iota.$
Again, it directly follows from \Cref{lemma:uniqueness_of_adjoints} that this category is a setoid.
Moreover, if we fix a pair of adjoints $L \dashv G \dashv R$, then any object in that category is isomorphic to $(L \dashv G \dashv R, R \xrightarrow{\iota} L)$ where $\iota \in \Iso( R, L )$ is an arbitrary isomorphism. Note that $\Iso( R, L ) \cong \Aut( R )$. In this sense, the two notions of defining an internal ambiadjunction are equivalent.
\end{remark}

\begin{notation}
Let $\cC$ denote a monoidal category. We denote by 
\begin{itemize}
    \item $\Cat$ the strict bicategory of categories,
    \item $\rMod{\cC}$ the strict bicategory of right $\cC$-modules (see \cite{EGNO} or \cite{FLP2}*{A.2}, we also briefly recall this bicategory in \Cref{section:monoidal_induce_ambi_of_bim}),
    \item $\lMod{\cC}$ the strict bicategory of left $\cC$-modules (we also briefly recall this bicategory in \Cref{section:monoidal_induce_ambi_of_bim}),
    \item $\BiMod{\cC}$ the strict bicategory of $\cC$-bimodules (see \cite{FLP2}*{A.3} for details, we also briefly recall this bicategory in \Cref{section:monoidal_induce_ambi_of_bim}),
    \item $\Catlax$ the strict bicategory of monoidal categories and lax monoidal functors (see \cite{FLP2}*{Section~2.2.3}),
    \item $\Catoplax$  the strict bicategory of monoidal categories and oplax monoidal functors.
\end{itemize}
In this paper, we will make use of the notion of an internal (ambi)adjunction (see \Cref{def:adjunction} and \Cref{def:ambiadj}) in  these strict bicategories. We call an (ambi)adjunction internal to
\begin{itemize}
    \item $\Cat$ an \emph{(ambi)adjunction} of categories,
    \item $\rMod{\cC}$ an \emph{(ambi)adjunction} of right $\cC$-modules,
    \item $\lMod{\cC}$ an \emph{(ambi)adjunction} of left $\cC$-modules,
    \item $\BiMod{\cC}$ an \emph{(ambi)adjunction} of $\cC$-bimodules,
    \item $\Catlax$ a monoidal \emph{(ambi)adjunction},
    \item $\Catoplax$ an opmonoidal \emph{(ambi)adjunction}.
\end{itemize}
\end{notation}

\begin{remark}\label{remark:strong_implies_monoidal}
Let $G: \cC \rightarrow \cD$ be a strong monoidal functor.
If $R: \cD \rightarrow \cC$ is a functor such that $G \dashv R$ is an adjunction of categories, then there exists a uniquely determined lax structure $\lax^R$ on $R$ such that $G \dashv R$ is a monoidal adjunction (see, e.g., \cite{FLP2}*{Lemma 2.7}).
Dually, if $L: \cD \rightarrow \cC$ is a functor such that $L \dashv G$ is an adjunction of categories, then there exists a uniquely determined oplax structure $\oplax^L$ on $L$ such that $L \dashv G$ is an opmonoidal adjunction. 
\end{remark}

\begin{remark}\label{remark:oplax_lax_adjunction}
Let $\cC$ and $\cD$ be monoidal categories.
Let $G: \cC \rightarrow \cD$ be an oplax monoidal functor and let $R: \cD \rightarrow \cC$ be a lax monoidal functor.
Let $G \dashv R$ be an adjunction of categories.
If $G \dashv R$ is a monoidal adjunction, then $G$ is already strong monoidal.
If $G \dashv R$ is an opmonoidal adjunction, then $R$ is already strong monoidal.
However, there is also the notion of an \emph{oplax-lax adjunction} (cf.~\cite{AM}*{Section~3.9.1.}), which is characterized by the commutativity of the following diagrams:
\begin{equation}\label{eq:oplax_lax}
    \begin{tikzpicture}[baseline=($(11) + 0.5*(d)$)]
          \coordinate (r) at (7,0);
          \coordinate (d) at (0,-2);
          \node (11) {$\objCa \otimes \objCb$};
          \node (12) at ($(11) + (r)$) {$\rightadj\mainfun(\objCa \otimes \objCb)$};
          \node (21) at ($(11) + (d)$) {$\rightadj\mainfun(\objCa) \otimes \rightadj\mainfun(\objCb)$};
          \node (22) at ($(11) + (d) + (r)$) {$\rightadj(\mainfun(\objCa) \otimes \mainfun(\objCb))$};
          \draw[->] (11) to node[above]{$\unit_{\objCa \otimes \objCb}$} (12);
          \draw[->] (21) to node[below]{$\lax_{\mainfun(\objCa), \mainfun(\objCb)}$} (22);
          \draw[->] (11) to node[left]{$\unit_{\objCa} \otimes \unit_{\objCb}$} (21);
          \draw[->] (12) to node[right]{$\rightadj( \oplax_{\objCa, \objCb} )$} (22);
    \end{tikzpicture}    
\end{equation}
and
\begin{equation}\label{eq:unit_lax}
    \begin{tikzpicture}[baseline=($(11) + 0.5*(d)$)]
          \coordinate (r) at (4,0);
          \coordinate (d) at (0,-2);
          \node (11) {$\one$};
          \node (12) at ($(11) + (r)$) {$\rightadj\mainfun(\one)$};
          \node (21) at ($(11) + (d)$) {$\one$};
          \node (22) at ($(11) + (d) + (r)$) {$\rightadj(\one)$};
          \draw[->] (11) to node[above]{$\unit_{\one}$} (12);
          \draw[->] (21) to node[below]{$\lax_{0}$} (22);
          \draw[->] (11) to node[left]{$\id$} (21);
          \draw[->] (12) to node[right]{$\rightadj( \oplax_{0} )$} (22);
    \end{tikzpicture}    
\end{equation}
for $A, B \in \cC$.
In an oplax-lax adjunction, neither $G$ nor $R$ need to be strong. In that sense, an oplax-lax adjunction generalizes both monoidal and opmonoidal adjunctions.
\end{remark}

\section{From Frobenius functors to Frobenius monoidal functors}

\subsection{Frobenius monoidal functors between endomorphism categories}
\label{sec:FrobEndo}

In \Cref{sec:prelim}, we have recalled the notions of a Frobenius monoidal functor and of a Frobenius functor. These two notions are claimed to be unrelated (see \cite{nlab:frobenius_monoidal_functor}). Indeed, the notion of a Frobenius monoidal functor only makes sense in the context of monoidal categories, whereas the notion of a Frobenius functor makes sense in the context of an arbitrary bicategory.

In this section, we show that at least Frobenius functors induce Frobenius monoidal functors between endomorphism categories. We recall that endomorphism categories are always monoidal categories, and hence, it makes sense to look for Frobenius monoidal functors between them.

\begin{definition}\label{definition:monoidal_end}
Let $\cC$ be an object in a strict bicategory.
We denote by $\End( \cC )$ its strict monoidal \emph{endomorphism category}. The objects are given by endomorphisms $\cC \rightarrow \cC$. The morphisms are given by $2$-morphisms of such endomorphisms. The tensor product is given by composition of endomorphisms. The tensor unit is given by $\id_{\cC}$.
\end{definition}

\begin{proposition}\label{proposition:induced_end_adjunctions}
Let
\[
\begin{tikzpicture}[mylabel/.style={fill=white}]
      \coordinate (r) at (3,0);
      \node (A) {$\cC$};
      \node (B) at ($(A) + (r)$) {$\cD$};
      \draw[->, out = 30, in = 180-30] (A) to node[mylabel]{$G$} (B);
      \draw[<-, out = -30, in = 180+30] (A) to node[mylabel]{$F$} (B);
\end{tikzpicture} 
\]
be morphisms in a strict bicategory.
Consider the functors between endomorphism categories
\[
\begin{tikzpicture}[mylabel/.style={fill=white}]
      \coordinate (r) at (5,0);
      \node (A) {$\End(\cC)$};
      \node (B) at ($(A) + (r)$) {$\End(\cD)$};
      \draw[->, out = 15, in = 180-15] (A) to node[mylabel]{$\Gamma := \Hom( F, G )$} (B);
      \draw[<-, out = -15, in = 180+15] (A) to node[mylabel]{$\Phi := \Hom( G, F )$} (B);
\end{tikzpicture} 
\]
given by
\[
\Gamma( A ) = \Hom(F,G)(A) = GAF
\]
and
\[
\Phi( X ) = \Hom(G,F)(X) = FXG
\]
for $A \in \End( \cC )$, $X \in \End( \cD )$.
Then an internal adjunction $G \dashv F$ induces an oplax-lax adjunction $\Gamma \dashv \Phi$ (see \Cref{remark:oplax_lax_adjunction}) with the following structure morphisms for $A,B \in \End( \cC )$, $X,Y \in \End( \cD )$:

\textbf{Lax structure on $\Phi$:}
\begin{align*}
\lax^{\Phi}_0 &:= \big(\id_{\cC} \xrightarrow{\unit^{G \dashv F}} FG = \Phi( \id_{\cD} )\big) \\
\lax^{\Phi}_{X,Y} &:= \big(\Phi(X)\Phi(Y) = FXGFYG \xrightarrow{FX \ast \counit^{G \dashv F} \ast YG} FXYG = \Phi(XY)\big)
\end{align*}

\textbf{Oplax structure on $\Gamma$:}
\begin{align*}
\oplax^{\Gamma}_0 &:= \big( \Gamma( \id_{\cC} ) = GF \xrightarrow{\counit^{G \dashv F}} \id_{\cD} \big) \\
\oplax^{\Gamma}_{A,B} &:= \big( \Gamma(AB) = GABF \xrightarrow{GA \ast \unit^{G \dashv F} \ast BF} GAFGBF = \Gamma(A)\Gamma(B)\big)
\end{align*}

\textbf{Unit of the adjunction $\Gamma \dashv \Phi$:}
\begin{align*}
    \unit^{\Gamma \dashv \Phi} := \big( A \xrightarrow{\unit^{G \dashv F} \ast A \ast \unit^{G \dashv F}} FGAFG = \Phi\Gamma A\big)
\end{align*}

\textbf{Counit of the adjunction $\Gamma \dashv \Phi$:}
\begin{align*}
    \counit^{\Gamma \dashv \Phi} := \big( \Gamma\Phi(X) = GFXGF \xrightarrow{\counit^{G \dashv F} \ast X \counit^{G \dashv F}} X\big)
\end{align*}
\end{proposition}

\begin{proof}
Suppose given an internal adjunction $G \dashv F$.
We show that the unitality constraints hold for $\lax^{\Phi}$. We compute for $X \in \End( \cD )$:
\begin{align*}
&\phantom{=}\Phi(X) \xrightarrow{\lax_0^{\Phi} \ast \id_{\Phi(X)}} \Phi(\id_{\cD})\Phi(X) \xrightarrow{\lax^{\Phi}_{\id_{\cD},X}} \Phi(X) & &\\
&= FXG \xrightarrow{\unit^{G \dashv F} \ast FXG} FGFXG \xrightarrow{F \ast \counit^{G \dashv F} \ast XG} FXG & & \text{(def)}\\
&= \id_{\Phi(X)} & & \text{(zig zag)}
\end{align*}
The second unitality constraint holds by a similar computation.
Next, we show that the associativity constraint holds for $\lax^{\Phi}$. We compute for $X, Y, Z \in \End( \cD )$:
\begin{align*}
&\phantom{=}\Phi(X)\Phi(Y)\Phi(Z) \xrightarrow{\id \ast \lax_{Y,Z}} \Phi(X)\Phi(YZ)\xrightarrow{\lax_{X,YZ}} \Phi(XYZ) & &\\
&=FXGFYGFZG \xrightarrow{FXGFY\counit^{G \dashv F}ZG} FXGFYZG \xrightarrow{FX\counit^{G \dashv F} YZG} FXYZG & & \text{(def)}
\end{align*}
which is equal to
\begin{align*}
&\phantom{=}\Phi(X)\Phi(Y)\Phi(Z) \xrightarrow{\lax_{XY,Z}} \Phi(XY)\Phi(Z)\xrightarrow{\lax_{XY,Z}} \Phi(XYZ) & &\\
&=FXGFYGFZG \xrightarrow{FX\counit^{G\dashv F}YGFZG} FXYGFZG \xrightarrow{FXY\counit^{G \dashv F}ZG} FXYZG & & \text{(def)}
\end{align*}
by the interchange law. Thus, $\lax^{\Phi}$ defines a lax structure on $\Phi$.
Similarly, $\oplax^{\Gamma}$ defines an oplax structure on $\Gamma$.

Next, we compute the zigzag identities of $\Gamma \dashv \Phi$:
\begin{align*}
&\phantom{=}\Gamma( A ) \xrightarrow{\Gamma(\unit_A^{\Gamma \dashv \Phi})} \Gamma\Phi\Gamma(A) \xrightarrow{\counit_{\Gamma(A)}^{\Gamma \dashv \Phi}} \Gamma(A) \\
&= GAF \xrightarrow{G \ast \unit^{G \dashv F} \ast A \ast \unit^{G \dashv F} \ast F} GFGAFGF \xrightarrow{\counit^{G \dashv F} \ast GAF \ast \counit^{G \dashv F}} GAF \\
&= \id_{\Gamma(A)}
\end{align*}
where the last equality holds due to the zigzag identity of $G \dashv F$. Similarly, the other zigzag identity holds for $\Gamma \dashv \Phi$.

Last, we compute the compatibilities \Cref{eq:unit_lax} and \Cref{eq:oplax_lax} that turn $\Gamma \dashv \Phi$ into an oplax-lax adjunction:

\begin{align*}
&\phantom{=} \id_{\cC} \xrightarrow{\unit_{\id_{\cC}}^{\Gamma \dashv \Phi}} \Phi\Gamma( \id_{\cC} ) \xrightarrow{\Phi( \oplax_0^{\Gamma} )} \Phi( \id_{\cC} ) & \\
&= \id_{\cC} \xrightarrow{\unit^{G \dashv F} \ast \unit^{G \dashv F}} FGFG \xrightarrow{F \ast \counit^{G \dashv F} \ast G} FG & \\
&= \id_{\cC} \xrightarrow{\unit^{G \dashv F}} FG \xrightarrow{\unit^{G \dashv F} \ast F \ast G} FGFG \xrightarrow{F \ast \counit^{G \dashv F} \ast G} FG & \text{(interchange)}\\
&= \id_{\cC} \xrightarrow{\unit^{G \dashv F}} FG \xrightarrow{\id_{F} \ast G} FG & \text{(zigzag)}\\
&= \id_{\cC} \xrightarrow{\lax_0^{\Phi}} FG & \text{(def)}
\end{align*}
and
\begin{align*}
&\phantom{=} AB \xrightarrow{\unit^{\Gamma \dashv \Phi}_A \ast \unit^{\Gamma \dashv \Phi}_B} \Phi\Gamma(A) \Phi\Gamma(B) \xrightarrow{\lax^{\Phi}_{\Gamma(A), \Gamma(B)}} \Phi( \Gamma(A)\Gamma(B) ) & \\
&= AB \xrightarrow{\unit^{G \dashv F} \ast A \ast (\unit^{G \dashv F} \ast \unit^{G \dashv F}) \ast B \ast \unit^{G \dashv F}} FGAFG FGBFG \\&\phantom{= AB}\xrightarrow{FGA(F \ast \counit^{G \dashv F} \ast G)BFG} FGAFGBFG & \text{(def)}\\
&= AB \xrightarrow{\unit^{G \dashv F} \ast A \ast \unit^{G \dashv F} \ast B \ast \unit^{G \dashv F}} FGAFGBFG & \text{(zigzag)}\\
&= AB \xrightarrow{\unit^{G \dashv F} \ast AB \ast \unit^{G \dashv F}} FGABFG \xrightarrow{FGA \ast \unit^{G \dashv F} \ast BFG} FGAFGBFG & \text{(interchange)}\\
&= AB \xrightarrow{\unit^{\Gamma \dashv \Phi}_{AB}} \Phi\Gamma(AB) \xrightarrow{\Phi( \oplax^{\Gamma}_{A,B} )} \Phi( \Gamma(A)\Gamma(B) )
\end{align*}
This proves the claim.
\end{proof}

\begin{theorem}[A relationship between Frobenius monoidal functors and Frobenius functors]\label{theorem:relation_frobenius_monoidal_and_ambi}
In the context of \Cref{proposition:induced_end_adjunctions}: if $G$ is a Frobenius functor, i.e., if we have an internal ambiadjunction $F \dashv G \dashv F$, then $(\Phi, \lax^{\Phi}, \oplax^{\Phi})$ and $(\Gamma, \lax^{\Gamma}, \oplax^{\Gamma})$ are Frobenius monoidal functors.
\end{theorem}
\begin{proof}
By symmetry, it suffices to prove the claim for $\Phi$.
We show that $\Phi$ satisfies \Cref{frobmon1}. For $X,Y,Z \in \End( \cD )$, we have:
\begin{align*}
&\phantom{=} \Phi(X) \Phi(YZ) \xrightarrow{\Phi(X) \ast \oplax^{\Phi}_{Y,Z}} \Phi(X)\Phi(Y)\Phi(Z) \xrightarrow{\lax^{\Phi}_{X,Y} \ast \Phi(Z)} \Phi(XY)\Phi(Z) \\
&= FXG FYZG \xrightarrow{FXG FY\unit^{F \dashv G} ZG} FXG FYG FZG \xrightarrow{FX\counit^{G \dashv F}YG ZG} FXYG FZG \\
&= FXG FYZG \xrightarrow{FX\counit^{G \dashv F}YZG} FXYZG \xrightarrow{FXY\unit^{F \dashv G} ZG} FXYG FZG \\
&= \Phi(X) \Phi(YZ) \xrightarrow{\lax^{\Phi}_{X,YZ}} \Phi(XYZ) \xrightarrow{\oplax^{\Phi}_{XY,Z}} \Phi(XY)\Phi(Z)
\end{align*}
In this computation, the second equation follows by the interchange law.
The computation that $\Phi$ satisfies \Cref{frobmon2} is similar.
\end{proof}

\begin{example}\label{ex-classical-Frob}
Every monoidal category $\cC$ can be regarded as a bicategory $\mathbf{B}(\cC)$ with a single object \cite{JY21}*{Example 2.1.19}: $1$-morphisms of $\mathbf{B}(\cC)$ are the objects of $\cC$, their composition is given by the tensor product of $\cC$, and $2$-morphisms of $\mathbf{B}(\cC)$ are the morphisms of $\cC$.
An internal ambiadjunction in $\mathbf{B}(\cC)$ corresponds to objects $A, B \in \cC$ where $B$ is both a left and a right dual of $A$.
Diagrammatically, we have
\begin{center}
    \begin{tikzpicture}[ baseline=(A)]
          \coordinate (r) at (3.5,0);
          \node (A) {$\bullet$};
          \node (B) at ($(A) + (r)$) {$\bullet$};
          \draw[->, out = 30, in = 180-30] (A) to node[above]{$A$} (B);
          \draw[<-, out = -30, in = 180+30] (A) to node[below]{$B$} (B);
          \node[rotate=90] (t) at ($(A) + 0.45*(r)$) {$\vdash$};
          \node[rotate=-90] (t) at ($(A) + 0.55*(r)$) {$\vdash$};
    \end{tikzpicture}  
\end{center}
where $\bullet$ denotes the single object of $\mathbf{B}(\cC)$.

Now, we have that $\End_{\mathbf{B}(\cC)}( \bullet ) \simeq \cC$ as monoidal categories. Thus, when we apply \Cref{theorem:relation_frobenius_monoidal_and_ambi} in this situation, we obtain an ambiadjunction
\begin{center}
    \begin{tikzpicture}[ baseline=(A)]
          \coordinate (r) at (3.5,0);
          \node (A) {$\cC$};
          \node (B) at ($(A) + (r)$) {$\cC$};
          \draw[->, out = 30, in = 180-30] (A) to node[above]{$\Phi( C ) = A \otimes C \otimes B$} (B);
          \draw[<-, out = -30, in = 180+30] (A) to node[below]{$\Psi( C ) = B \otimes C \otimes A$} (B);
          \node[rotate=90] (t) at ($(A) + 0.45*(r)$) {$\vdash$};
          \node[rotate=-90] (t) at ($(A) + 0.55*(r)$) {$\vdash$};
    \end{tikzpicture}  
\end{center}
in which both functors are Frobenius monoidal functors. Concretely, the lax structure on $\Phi$ is given by
\[
A \otimes C \otimes B \otimes A \otimes D \otimes B \xrightarrow{A \otimes C \otimes \ev \otimes D \otimes B} A \otimes C \otimes D \otimes B
\]
and the oplax structure on $\Phi$ is given by
\[
A \otimes C \otimes D \otimes B
\xrightarrow{A \otimes C \otimes \coev \otimes D \otimes B} 
A \otimes C \otimes B \otimes A \otimes D \otimes B
\]
for $C,D \in \cC$, where $\ev: B \otimes A \rightarrow \one$ denotes the evaluation for $B$ as a left dual of $A$ and $\coev: \one \rightarrow B \otimes A$ the coevaluation for $B$ as a right dual for $A$.
In particular, the objects $A \otimes B$ and $B \otimes A$ can be regarded as Frobenius algebras internal to $\cC$.
\end{example}

\subsection{A context for the projection formulas in an ambiadjunction}
\label{sec:context-proj}

It is the goal of this short section to introduce and explain \Cref{context:main}, which will be the context for our main theorems in \Cref{sec:F-Frob-mon} and \Cref{section:functors_on_centers}.

For an oplax-lax adjunction $G\dashv R$ (see \Cref{remark:oplax_lax_adjunction}), consider the following natural transformations:
\begin{gather}\label{eq:proj}
\projl{\objCa}{\objDx}^{R, G \dashv R}\colon \objCa \otimes \rightadj\objDx \xrightarrow{\unit_{\objCa}^{G \dashv R} \otimes \id} \rightadj\mainfun(\objCa) \otimes \rightadj\objDx \xrightarrow{\lax_{\mainfun\objCa,\objDx}^R} \rightadj( \mainfun\objCa \otimes \objDx )
\\\label{eq:projr}
\projr{\objDx}{\objCa}^{R, G \dashv R}\colon \rightadj\objDx \otimes \objCa \xrightarrow{\id \otimes \unit_{\objCa}^{G \dashv R} } \rightadj\objDx \otimes \rightadj\mainfun(\objCa)  \xrightarrow{\lax_{X,GA}^R} \rightadj(  \objDx \otimes \mainfun\objCa )
\end{gather}
for $A \in \cC$, $X \in \cD$.
We call $\projlnoarg^{R, G \dashv R}$ and $\projrnoarg^{R, G \dashv R}$ the (\emph{left} or \emph{right}, respectively) \emph{projection formula morphism (of $R$ in the context of $G \dashv R$)}.

Dually, for an oplax-lax adjunction $L\dashv G$, consider the following natural transformations:
\begin{gather}\label{eq:iproj}
\projl{\objCa}{\objDx}^{L, L \dashv G}\colon \leftadj( \mainfun\objCa \otimes \objDx ) \xrightarrow{\oplax_{\mainfun\objCa,\objDx}^L}\leftadj\mainfun(\objCa) \otimes \leftadj\objDx \xrightarrow{\counit_{\objCa}^{L \dashv G} \otimes \id}  \objCa \otimes \leftadj\objDx\\\label{eq:iprojr}
\projr{\objDx}{\objCa}^{L, L \dashv G}\colon \leftadj(  \objDx \otimes \mainfun\objCa )  \xrightarrow{\oplax_{\mainfun\objCa,\objDx}^L}  \leftadj\objDx \otimes \leftadj\mainfun(\objCa) \xrightarrow{\id \otimes \counit_{\objCa}^{L \dashv G} } \leftadj\objDx \otimes \objCa
\end{gather}
for $A \in \cC$, $X \in \cD$.
We call $\projlnoarg^{L, L \dashv G}$ and $\projrnoarg^{L, L \dashv G}$ the (\emph{left} or \emph{right}, respectively) \emph{projection formula morphism (of $L$ in the context of $L \dashv G$)}.

\begin{definition}\label{def:proj-formulas-hold}
We say that the 
\begin{itemize}
    \item \emph{left projection formula holds for $R$ (in the context of $G \dashv R$)} if $\projlnoarg^{R, G \dashv R}$ is an isomorphism,
    \item \emph{right projection formula holds for $R$ (in the context of $G \dashv R$)} if $\projrnoarg^{R, G \dashv R}$ is an isomorphism,
    \item \emph{projection formula holds for $R$ (in the context of $G \dashv R$)} if both $\projlnoarg^{R, G \dashv R}$ and $\projrnoarg^{R, G \dashv R}$ are isomorphisms.
\end{itemize}
Similarly, we define that the \emph{(left/right) projection formula holds for $L$ (in the context of $L \dashv G$)} depending on whether $\projlnoarg^{L, L \dashv G}$ and $\projrnoarg^{L, L \dashv G}$ are isomorphisms.  
\end{definition}

\begin{context}\label{context:main}
We describe a context which we will use for our main theorems in \Cref{sec:F-Frob-mon} and \Cref{section:functors_on_centers}:
\begin{itemize}
    \item $G: \cC \rightarrow \cD$ is a strong monoidal functor.
    \item $F: \cD \rightarrow \cC$ is a functor.
    \item We have an ambiadjunction $F\dashv G\dashv F$ of categories (i.e., internal to $\Cat$).
    \item We regard $G \dashv F$ as a monoidal adjunction and $F \dashv G$ as an opmonoidal adjunction by \Cref{remark:strong_implies_monoidal}. In particular, we both have a lax and an oplax structure on $F$.
\end{itemize}
\end{context}

\begin{definition}\label{definition:mutual_inverses}
    In \Cref{context:main}, we say that the
    \emph{projection formulas hold for $F$ regarded both as a left and as a right adjoint} if the morphisms $\projlnoarg^{F, F \dashv G}$, $\projrnoarg^{F, F \dashv G}$, $\projlnoarg^{F, G \dashv F}$, and $\projrnoarg^{F, G \dashv F}$ are isomorphisms. We say that the
    \emph{left projection formula morphisms are mutual inverses} if we have 
    \[\projlnoarg^{F, F \dashv G} = (\projlnoarg^{F, G \dashv F})^{-1}.\]
    We say that the \emph{right projection formula morphisms are mutual inverses} if we have 
    \[\projrnoarg^{F, F \dashv G} = (\projrnoarg^{F, G \dashv F})^{-1}.\]
    We say that the \emph{projection formula morphisms are mutual inverses} if both the left and the right projection formulas are mutually inverses.
\end{definition}

\begin{lemma}\label{lemma:braided_context}
In \Cref{context:main}, assume that the right or left projection formula morphisms are mutual inverses.  Additionally, assume that $(\cC,\Psi^\cC)$ and $(\cD,\Psi^\cD)$ are braided such that $G\colon \cC\to \cD$ is a braided functor.
Then the projection formula morphisms are mutual inverses.
\end{lemma}
\begin{proof}
By \cite{FLP2}*{Proposition~2.5 and Lemma~3.13}, the left projection formula morphism is obtained from the right projection formula morphism via the formula
\[
R( \Psi^{\cD}_{X,FA} ) \circ \projr{X}{A} = \projl{A}{X} \circ \Psi^{\cC}_{FX,A}
\]
for $A \in \cC$, $X \in \cD$.
This formula holds for both the monoidal adjunction $G\dashv F$ and the opmonoidal adjunction $F\dashv G$ by duality. If the right (left, resp.) projection formula morphisms are mutual inverses, then it follows that the left (right, resp.) projection formula morphisms are also mutual inverses.
\end{proof}

\subsection{When do monoidal functors induce ambiadjunctions of (bi)module categories?}\label{section:monoidal_induce_ambi_of_bim}

Let $\cC$ and $\cD$ be strict monoidal categories and let $G: \cC \rightarrow \cD$ be a strong monoidal functor. 

The goal of this section is to give a sufficient criterion for $G$ to become a Frobenius functor internal to $\rMod{\cC}$ and $\BiMod{\cC}$. The results of this section will later be of central importance in proving the main results of our paper.

Recall that a \emph{right $\cC$-module category} is a category $\cM$ equipped with a right action functor
$\triangleleft: \cM \times \cC \rightarrow \cM$
and an isomorphism 
$M \triangleleft (A \otimes B) \xrightarrow{\sim} (M \triangleleft A) \triangleleft B$
natural in $M \in \cM$, $A, B \in \cC$ which satisfy certain coherences, cf.\ \cite{EGNO}*{Section~7.1}.

A \emph{right $\cC$-module functor} between $\cC$-modules $\cM$ and $\cN$ is
a functor $F: \cM \rightarrow \cN$ between the underlying categories equipped with an isomorphism (called the \emph{right lineator})
\begin{align*}
\rlineator_{M,A}: F(M \triangleleft A) \xrightarrow{\sim} F(M) \triangleleft A
\end{align*}
natural in $A \in \cC$, $M \in \cM$.
Again, these data satisfy certain coherences, cf. \cite{EGNO}*{Definition~7.2.1} (for the left version). We denote a right $\cC$-module functor via the pair 
$(F, \rlineator)$.
A right $\cC$-module transformation between two right $\cC$-module functors is given by a natural transformation that commutes with the right lineators. We obtain a strict bicategory $\rMod{\cC}$ of right $\cC$-modules. Analogously, we obtain a strict bicategory of left $\cC$-modules $\lMod{\cC}$. We denote the left action functor of a left $\cC$-module $\cM$ by $\triangleright: \cC \times \cM \rightarrow \cM$.

Recall that a \emph{$\cC$-bimodule category} is a category $\cM$ equipped both with the structure of a left and right $\cC$-module together with an isomorphism 
$(A \triangleright M) \triangleleft B\isomorph A \triangleright (M \triangleleft B)$
natural in $A,B \in \cC$, $M \in \cM$. These data satisfy certain coherences, see \cite{FLP2}*{Definition A.15} for details.

In this paper, we will only make use of the following two kinds of $\cC$-(bi)module categories:

\begin{example}
The monoidal category $\cC$ can be equipped with left and right action functors given by the tensor product, i.e., 
\begin{align*}
A \triangleright B = A \otimes B & & B \triangleleft A = B \otimes A
\end{align*}
for $A, B \in \cC$. These actions and the associator turn $\cC$ into a $\cC$-(bi)module category.
\end{example}

\begin{definition}\label{def:DG}
If $G: \cC \rightarrow \cD$ is a strong monoidal functor, then $\cD$ can be equipped with left and right action functors given by 
\begin{align*}
A \triangleright X = G(A) \otimes X & & X \triangleleft A = X \otimes G(A)
\end{align*}
for $A \in \cC$, $X \in \cD$, see \cite{FLP2}*{Example A.22}. We denote\footnote{This notation is not to be confused with the equivariantization of a category with a group action, which is sometimes also denoted similarly.} the resulting $\cC$-(bi)module category by $\cD^G$. 
\end{definition}

A \emph{$\cC$-bimodule functor} between $\cC$-bimodules $\cM$ and $\cN$ is
a functor $F: \cM \rightarrow \cN$ between the underlying categories equipped with isomorphisms (called \emph{left/right lineators})
\begin{align*}
\llineator_{A,M}: F(A \triangleright M) \xrightarrow{\sim} A \triangleright F(M) & & \rlineator_{M,A}: F(M \triangleleft A) \xrightarrow{\sim} F(M) \triangleleft A
\end{align*}
natural in $A \in \cC$, $M \in \cM$.
These lineators turn $F$ into both a left and right $\cC$-module functor and satisfy a further coherence, see \cite{FLP2}*{Definition A.16} for details.
We denote a $\cC$-bimodule functor via the triple 
\[(F, \llineator, \rlineator).\]
A transformation between two $\cC$-bimodule functors is given by a natural transformation that commutes with both the left and the right lineators.
We obtain a strict bicategory $\BiMod{\cC}$ of $\cC$-bimodule categories.

\begin{example}\label{example:G_as_bimodule_functor}
A strong monoidal functor $G: \cC \rightarrow \cD$ gives rise to the $\cC$-bimodule functor
\[
\cC \xrightarrow{(G, \oplax^G, \oplax^G)} \cD^G,    
\]
i.e., its left and right lineators are given by the oplax structure of $G$:
\[
\llineator^G_{A,B}=\rlineator^G_{A,B}=\oplax^G_{A,B}\colon G(A\otimes B) \to GA\otimes GB.
\]
\end{example}

\begin{example}[\cite{FLP2}*{Section 3.4}]\label{example:R_as_module_functor}
If $G\dashv R$ is a monoidal adjunction such that the right projection formula holds for $R$ (see \Cref{def:proj-formulas-hold}), then 
\[
\cD^G\xrightarrow{(R, \rlineator^R)} \cC
\]
is a right $\cC$-module functor, with right lineator given by
\begin{align*}
\rlineator^R_{X,A}&=(\projr{X}{A}^{R, G \dashv R})^{-1}\colon  R(X\otimes GA)\to RX\otimes A
\end{align*}
for $A \in \cC$, $X \in \cD$.
Likewise, if the left projection formula holds for $R$, then
\[
\cD^G\xrightarrow{(R, \llineator^R)} \cC
\]
is a left $\cC$-module functor, with left lineator given by
\begin{align*}
\llineator^R_{A,X}=(\projl{A}{X}^{R, G \dashv R})^{-1}\colon  R(GA\otimes X)\to A\otimes RX
\end{align*}
for $A \in \cC$, $X \in \cD$.
If both the left and right projection formula holds for $R$, then we get a $\cC$-bimodule functor \cite{FLP2}*{Proposition~3.22}.
\end{example}

\begin{example}[Dual to \Cref{example:R_as_module_functor}]\label{example:L_as_module_functor}
If $L\dashv G$ is an opmonoidal adjunction such that the right projection formula holds for $L$, then
\[
\cD^G\xrightarrow{(L, \rlineator^L)} \cC
\]
is a right $\cC$-module functor, with right lineator given by
\begin{align*}
\rlineator^L_{X,A}=(\projr{X}{A}^{L, L \dashv G})\colon  L(X\otimes GA)\to LX\otimes A
\end{align*}
for $A \in \cC$, $X \in \cD$.
Likewise, if the left projection formula holds for $L$, then
\[
\cD^G\xrightarrow{(L, \llineator^L)} \cC
\]
is a left $\cC$-module functor, with left lineator given by
\begin{align*}
\llineator^L_{A,X}=(\projl{A}{X}^{L, L \dashv G})\colon  L(GA\otimes X)\to A\otimes LX
\end{align*}
for $A \in \cC$, $X \in \cD$.
If both the left and right projection formula holds for $L$, then we get a $\cC$-bimodule functor.
\end{example}

\begin{lemma}\label{lemma:module_adjunction_R}
Let $G\dashv R$ be a monoidal adjunction. 
\begin{enumerate}
    \item If the right projection formula holds for $R$, then we have an adjunction in $\rMod{\cC}$: \[(\cC \xrightarrow{(G, \oplax^G)} \cD^G) \dashv
(\cD^G\xrightarrow{(R, \rlineator^R)} \cC).\]
\item If the left projection formula holds for $R$, then we have an adjunction in $\lMod{\cC}$: \[(\cC \xrightarrow{(G, \oplax^G)} \cD^G) \dashv
(\cD^G\xrightarrow{(R, \llineator^R)} \cC).\]
\item If the projection formula holds for $R$, then we have an adjunction in $\BiMod{\cC}$: \[(\cC \xrightarrow{(G, \oplax^G, \oplax^G)} \cD^G) \dashv
(\cD^G\xrightarrow{(R, \llineator^R, \rlineator^R)} \cC).\]
\end{enumerate}
\end{lemma}
\begin{proof}
Follows from \cite{FLP2}*{Lemma 3.24} and its restriction to one-sided cases.
\end{proof}

\begin{lemma}[$\oop$-dual version of \Cref{lemma:module_adjunction_R}]\label{lemma:module_adjunction_L}
Let $L\dashv G$ be an opmonoidal adjunction.
\begin{enumerate}
    \item If the right projection formula holds for $L$, then we have an adjunction in $\rMod{\cC}$: \[(\cD^G\xrightarrow{(L, \rlineator^L)} \cC) \dashv
(\cC \xrightarrow{(G, \oplax^G, \oplax^G)} \cD^G).\]
\item If the left projection formula holds for $L$, then we have an adjunction in $\lMod{\cC}$: \[(\cD^G\xrightarrow{(L, \llineator^L)} \cC) \dashv
(\cC \xrightarrow{(G, \oplax^G, \oplax^G)} \cD^G).\]
\item If the projection formula holds for $L$, then we have an adjunction in $\BiMod{\cC}$: \[(\cD^G\xrightarrow{(L, \llineator^L, \rlineator^L)} \cC) \dashv
(\cC \xrightarrow{(G, \oplax^G, \oplax^G)} \cD^G).\]
\end{enumerate}
\end{lemma}

\begin{theorem}[From strong monoidal functors to ambiadjunctions in $\rMod{\cC}$]\label{theorem:strong_monoidal_to_ambiadjunction_right}
In \Cref{context:main}, assume that the right projection formulas hold for $F$ regarded both as a left and as a right adjoint.
Then the following statements are equivalent:
\begin{enumerate}
    \item The right projection formula morphisms are mutual inverses (\Cref{definition:mutual_inverses}).
    \item There exists a right lineator $\rlineator$ that turns $F$ into a right $\cC$-module functor of type $\cD^G \rightarrow \cC$ and gives rise to an ambiadjunction in $\rMod{\cC}$
\begin{center}
    \begin{tikzpicture}[ baseline=(A)]
          \coordinate (r) at (3.5,0);
          \node (A) {$\cC$};
          \node (B) at ($(A) + (r)$) {$\cD^G$};
          \draw[->, out = 30, in = 180-30] (A) to node[above]{$(G, \oplax^G)$} (B);
          \draw[<-, out = -30, in = 180+30] (A) to node[below]{$(F, \rlineator )$} (B);
          \node[rotate=90] (t) at ($(A) + 0.45*(r)$) {$\vdash$};
          \node[rotate=-90] (t) at ($(A) + 0.55*(r)$) {$\vdash$};
    \end{tikzpicture}    
\end{center}
whose underlying ambiadjunction of categories equals $F \dashv G \dashv F$.
\end{enumerate}
Moreover, if these statements hold, then we necessarily have $\rlineator = \projrnoarg^{F, F \dashv G}$.
\end{theorem}
\begin{proof}
By \Cref{lemma:module_adjunction_R} and \Cref{lemma:module_adjunction_L}, we have the following adjunctions in $\rMod{\cC}$:
\[
\big(G, \oplax^G\big) \dashv \big(F,(\projrnoarg^{F, G \dashv F})^{-1}\big) \qquad \text{and} \qquad \big(F, \projrnoarg^{F, F \dashv G}\big) \dashv \big(G, \oplax^G \big).
\]
Now, if the first statement holds, then we have an equality
\[
\big(F, (\projrnoarg^{F, G \dashv F})^{-1}\big) = \big(F,\projrnoarg^{F, F \dashv G}\big) 
\]
and hence the desired ambiadjunction in $\rMod{\cC}$ of the second statement.

Next, we show that the second statement implies the first statement.
By assumption, we have two right adjoints of 
$\big(G, \oplax^G\big)$ in $\rMod{\cC}$, namely
$(F,\rlineator)$ and $\big(F,(\projrnoarg^{F, G \dashv F})^{-1}\big)$.
From \Cref{lemma:uniqueness_of_adjoints} we obtain the comparison isomorphism
\[
\theta: (F,\rlineator) \isomorph \big(F,(\projrnoarg^{F, G \dashv F})^{-1}\big)
\]
in $\rMod{\cC}$. If we apply the forgetful pseudofunctor $\rMod{\cC} \rightarrow \Cat$ to $\theta$, then our assumptions imply that we obtain the comparison isomorphism that compares $F$ with itself. Hence, $\theta = \id_F$.
But the coherence conditions of $\id_F$ being a $\cC$-bimodule transformation already imply $\rlineator = (\projrnoarg^{F, G \dashv F})^{-1}$.
Likewise, we have two left adjoints of $\big(G, \oplax^G\big)$ in $\rMod{\cC}$, and the same line of arguments gives us $\rlineator = \projrnoarg^{F, F \dashv G}$. The first statement follows.
\end{proof}

\begin{theorem}[From strong monoidal functors to ambiadjunctions in $\BiMod{\cC}$]\label{theorem:strong_monoidal_to_ambiadjunction}
In \Cref{context:main}, assume that the projection formulas hold for $F$ regarded both as a left and as a right adjoint. Then the following statements are equivalent:
\begin{enumerate}
    \item The projection formula morphisms are mutual inverses (\Cref{definition:mutual_inverses}).
    \item There exist lineators $\llineator$ and $\rlineator$
that turn $F$ into a $\cC$-bimodule functor of type $\cD^G \rightarrow \cC$ and give rise to an ambiadjunction in $\BiMod{\cC}$
\begin{center}
    \begin{tikzpicture}[ baseline=(A)]
          \coordinate (r) at (3.5,0);
          \node (A) {$\cC$};
          \node (B) at ($(A) + (r)$) {$\cD^G$};
          \draw[->, out = 30, in = 180-30] (A) to node[above]{$(G, \oplax^G, \oplax^G)$} (B);
          \draw[<-, out = -30, in = 180+30] (A) to node[below]{$(F, \llineator, \rlineator )$} (B);
          \node[rotate=90] (t) at ($(A) + 0.45*(r)$) {$\vdash$};
          \node[rotate=-90] (t) at ($(A) + 0.55*(r)$) {$\vdash$};
    \end{tikzpicture}    
\end{center}
whose underlying ambiadjunction of categories equals $F \dashv G \dashv F$.
\end{enumerate}
Moreover, if these statements hold, then we necessarily have the equalities
\begin{align*}
\llineator = \projlnoarg^{F, F \dashv G}, & & \rlineator = \projrnoarg^{F, F \dashv G}.
\end{align*}
\end{theorem}
\begin{proof}
Combine \Cref{theorem:strong_monoidal_to_ambiadjunction_right} and its left-sided version.
\end{proof}

\subsection{Conditions for Frobenius functors to be Frobenius monoidal functors}
\label{sec:F-Frob-mon}

It is the goal of this section to prove the following first main theorem of the paper.

\begin{theorem}\label{theorem:frobenius_functors_frobenius_monoidal}
In \Cref{context:main}, assume that the right projection formula morphisms are mutual inverses. Then $(F, \lax^F, \oplax^F)$ is a Frobenius monoidal functor.
\end{theorem}

We start with the following immediate consequence of \Cref{theorem:relation_frobenius_monoidal_and_ambi} and \Cref{theorem:strong_monoidal_to_ambiadjunction_right}.

\begin{proposition}\label{proposition:phi_frobenius_right_mod}
In \Cref{context:main}, assume that the right projection formula morphisms are mutual inverses. Then we get Frobenius monoidal functors
\[
\begin{tikzpicture}[mylabel/.style={fill=white}]
      \coordinate (r) at (8,0);
      \node (A) {$\End_{\rMod{\cC}}(\cC)$};
      \node (B) at ($(A) + (r)$) {$\End_{\rMod{\cC}}(\cD^G)$};
      \draw[->, out = 15, in = 180-15] (A) to node[mylabel]{$\Gamma = \Hom( F, G )$} (B);
      \draw[<-, out = -15, in = 180+15] (A) to node[mylabel]{$\Phi = \Hom( G, F )$} (B);
\end{tikzpicture} 
\]
and oplax-lax adjunctions $\Gamma \dashv \Phi$ and $\Phi \dashv \Gamma$.
\end{proposition}

\begin{remark}\label{remark:ZprimeDG_explicitly}
We describe the category $\End_{\rMod{\cC}}(\cD^G)$ explicitly. Its objects are endofunctors $K: \cD \rightarrow \cD$
equipped with isomorphisms (right lineators)
$\rlineator_{X,A}: K(X \otimes G(A)) \isomorph K(X) \otimes G(A)$
natural in $A \in \cC$, $X \in \cD$ such that the following diagrams commute for $A, B \in \cC$, $X \in \cD$:

\begin{equation}\label{equation:K_rlinor_coh}
   \begin{tikzpicture}[baseline=($(21)$)]
  \coordinate (r) at (4,0);
  \coordinate (d) at (0,-2.5);
  \node (11) {$K( X \otimes  G(A\otimes B) )$};
  \node (21) at ($(11) +(d) - (r)$) {$K( X \otimes GA \otimes GB )$};
  \node (22) at ($(11) + (d) + (r)$) {$K(X) \otimes G(A \otimes B) $};
  \node (31) at ($(21) + (d)$) {$ K(X \otimes GA )  \otimes GB$};
  \node (32) at ($(22) + (d)$) {$K(X) \otimes GA\otimes GB$};
  \draw[->] (11) to node[above left]{$K(X \otimes \oplax^G_{A,B} )$} (21);
  \draw[->] (11) to node[above right]{$\rlineator_{X, A\otimes B}$} (22);
  \draw[->] (21) to node[left]{$\rlineator_{X \otimes GA,B}$} (31);
  \draw[->] (22) to node[right]{$K(X) \otimes \oplax^G_{A,B} $} (32);
  \draw[->] (31) to node[below]{$\rlineator_{X,A} \otimes GB$} (32);
\end{tikzpicture} 
\end{equation}

\begin{equation}\label{equation:K_rlinor_one_coh}
    \begin{tikzpicture}[baseline=($(11) + 0.5*(d)$)]
      \coordinate (r) at (4,0);
      \coordinate (d) at (0,-2);
      \node (11) {$K( X \otimes G\one )$};
      \node (12) at ($(11) + 2*(r)$) {$ K(X) \otimes G\one$};
      \node (2) at ($(11) + (d) + (r)$) {$K(X)$};
      \draw[->] (11) to node[above]{$\rlineator_{X, \one}$} (12);
      \draw[->] (11) to node[left, xshift=-0.5em]{$K(X \otimes \oplax^G_{\one} )$} (2);
      \draw[->] (12) to node[right, xshift=0.5em]{$K(X) \otimes \oplax^G_{\one} $} (2);
    \end{tikzpicture} 
\end{equation}

A morphism from $(K,\rlineator^K)$ to $(M, \rlineator^M)$ is given by a natural transformation $\tau: K \rightarrow M$ such that the following diagram commutes for $A \in \cC$, $X \in \cD$:
\begin{equation}\label{equation:tau_rlin_coh}
\begin{tikzpicture}[ baseline=($(11) + 0.5*(d)$)]
      \coordinate (r) at (7,0);
      \coordinate (d) at (0,-2);
      \node (11) {$K( X \otimes GA)$};
      \node (12) at ($(11) + (r)$) {$K(X) \otimes GA$};
      \node (21) at ($(11) + (d)$) {$M(X \otimes GA)$};
      \node (22) at ($(11) + (d) + (r)$) {$M(X) \otimes GA$};
      \draw[->] (11) to node[above]{$\rlineator^K_{X,A}$} (12);
      \draw[->] (21) to node[below]{$\rlineator^M_{X,A}$} (22);
      \draw[->] (11) to node[left]{$\tau_{X \otimes GA}$} (21);
      \draw[->] (12) to node[right]{$\tau_{X} \otimes GA$} (22);
\end{tikzpicture}
\end{equation}
\end{remark}

Now, we compare the monoidal categories $\cD$ with $\End_{\rMod{\cC}}(\cD^G)$ and  $\cC$ with $\End_{\rMod{\cC}}(\cC)$.

\begin{lemma}\label{lemma:DvsEndrightDG}
Let $G\colon \cC\to \cD$ be a strong monoidal functor. Then we have a strong monoidal functor
\begin{align*}
\Emb\colon \cD &\to \End_{\rMod{\cC}}(\cD^G) \\
X &\mapsto ( X\otimes(-), \rlineator^{(X \otimes -)} )
\end{align*}
where
\[
\rlineator^{(X \otimes -)}_{Y,A}: X \otimes Y \otimes GA \xrightarrow{\id} X \otimes Y \otimes GA
\]
for $A \in \cC$, $Y \in \cD$. Moreover, the underlying natural isomorphism of
\[
\lax^{\Emb}: \Emb( X )\Emb( Y ) \rightarrow \Emb( X \otimes Y )
\]
is given by the identity of $(X \otimes Y \otimes -): \cD \rightarrow \cD$ for $X,Y \in \cD$, and the underlying natural isomorphism of
\[
\lax_0^{\Emb}: \id_{\cD^G} \rightarrow \Emb( \one_{\cD} )
\]
is given by the identity of $\id_{\cD^G}$.
\end{lemma}
\begin{proof}
    First, one has to check that $(X\otimes(-),\rlineator^{X\otimes -})$ satisfies the coherences of a right $\cC$-module endofunctor of $\cD^G$. However, these coherences \eqref{equation:K_rlinor_coh} and\eqref{equation:K_rlinor_one_coh} degenerate to simple equalities of morphisms and hold trivially.
Moreover, for any morphism $f\colon X\to Y$ in $\cD$, one checks that $(f \otimes -): (X \otimes -) \rightarrow (Y \otimes -)$ satisfies the coherence  \eqref{equation:tau_rlin_coh} which also amounts to a simple equality of morphisms and which holds trivially.

Last, the identity of $(X \otimes Y \otimes -)$ turns $\Emb$ into a strong monoidal functor.
This is clear as the right lineators of $\Emb( X)\Emb( Y)$ and 
$\Emb( X\otimes Y )$ are both given by identity morphisms and we have that $\id_{\cD^G} = \Emb( \one_{\cD} )$ as right $\cC$-module endomorphisms of $\cD^G$. Thus, the coherences for $\Emb$ being a strong monoidal functor hold trivially.
\end{proof}

\begin{remark}\label{remark:C_equivalent_to_EndC}
It is well-known that if $G$ is the identity functor of $\cC$, then \Cref{lemma:DvsEndrightDG} gives us a monoidal equivalence
\[
\cC \simeq \End_{\rMod{\cC}}(\cC).
\]
See, e.g., \cite{EGNO}*{Proof of Theorem 2.8.5}.
\end{remark}

\begin{remark}\label{remark:rproj_as_a_module_morphism}
Let $G \dashv R$ be a monoidal adjunction such that the projection formula holds for $R$ and let $X \in \cD$.
Then $R(X \otimes G(-)): \cC \rightarrow \cC$ can be regarded as an object in $\End_{\rMod{\cC}}(\cC)$ with right lineator given as follows ($A,B \in \cC$):
\[
R(X \otimes G(A \otimes B)) \xrightarrow{R(X \otimes \oplax^G_{A,B})} R(X \otimes G(A) \otimes G(B)) \xrightarrow{(\projr{X \otimes GA}{B}^{R, G \dashv R})^{-1}} R(X \otimes G(A)) \otimes B.
\]
Moreover, we have the following isomorphism in $\End_{\rMod{\cC}}(\cC)$ (see \cite{FLP2}*{Lemma 3.4}):
\[
\beta_{X}: (RX \otimes -) \xrightarrow{\projr{X}{-}^{R, G \dashv R}} R(X \otimes G(-)),
\]
i.e., $R(X \otimes G(-))$ corresponds to the object $RX$ via the equivalence of \Cref{remark:C_equivalent_to_EndC}.
\end{remark}

\begin{lemma}\label{lemma:compare_R_Phi}
Let $G \dashv R$ be a monoidal adjunction such that the projection formula holds for $R$.
Then we obtain the following diagram in $\Catlax$
\begin{center}
\begin{tikzpicture}[baseline=($(D) + 0.5*(d)$)]
      \coordinate (r) at (10,0);
      \coordinate (d) at (0,-1.5);
      
      \node (D) {$\cD$};
      \node (C2) at ($(D) + (r)$) {$\cC$};
      \node (Dp) at ($(D) + (d)$) {$\End_{\rMod{\cC}}( \cD^G )$};
      \node (Cp2) at ($(Dp) + (r)$) {$\End_{\rMod{\cC}}( \cC )$};
      
      \draw[->] (D) to node[above]{$R$} (C2);

      \draw[->] (Dp) to node[below]{$\Phi = \Hom( G, R )$} (Cp2);

      \draw[->] (D) to  node[left]{$\Emb$} (Dp);
      \draw[->] (C2) to node[right]{$\Emb$} (Cp2);

      \draw[-implies,double equal sign distance,shorten >=3em,shorten <=3em] (C2) to node[below]{$\beta$} (Dp);
\end{tikzpicture}
\end{center}
where $\beta$ is a $2$-isomorphism in $\Catlax$ (i.e., a monoidal natural isomorphism) given by
\[
\beta_{X}: (RX \otimes -) \xrightarrow{\projr{X}{-}^{R, G \dashv R}} R(X \otimes G(-))
\]
for $X \in \cD$.
\end{lemma}
\begin{proof}
The two embedding functors are strong monoidal by \Cref{lemma:DvsEndrightDG}.
The functor $\Phi$ is lax by \Cref{proposition:induced_end_adjunctions}.
The functor $R$ is lax since $G \dashv R$ is a monoidal adjunction.
Thus, all functors lie in $\Catlax$.

Moreover, $\beta$ is compatible with the right lineators of its source and range by \Cref{remark:rproj_as_a_module_morphism}. Next, we take a look at the following diagram
\begin{equation}\label{equation:beta_monoidal}
  \begin{tikzpicture}[ baseline=(21)]
    \coordinate (r) at (8,0);
    \coordinate (d) at (0,-2);
    
    \node (11) {$RX \otimes RY \otimes A$};
    \node (12) at ($(11)+(r)$) {$R( X \otimes Y ) \otimes A$};

    \node (21) at ($(11) + (d)$){$RX \otimes R(Y \otimes G(A))$};
    
    \node (31) at ($(21) + (d)$){$R( X \otimes G(R(Y \otimes G(A))))$};
    \node (32) at ($(31)+(r)$) {$R(X \otimes Y \otimes G(A))$};
    
    \draw[->,thick] (11) to node[above]{$\lax^R_{X,Y} \otimes A$}(12);
    \draw[->,thick] (31) to node[below]{$R(X \otimes \counit_{Y \otimes G(A)})$}(32);

    \draw[->,thick] (11) to node[left]{$RX \otimes \projr{Y}{A}$}(21);
    \draw[->,thick] (21) to node[left]{$\projr{X}{R(Y \otimes GA)}$}(31);
    \draw[->,thick] (12) to node[right]{$\projr{X \otimes Y}{A}$}(32);
    
    \draw[->,thick] (21) to node[above]{$\lax^R_{X, Y \otimes GA}$}(32);
  \end{tikzpicture}
\end{equation}
where $A \in \cC$, $X,Y \in \cD$. Note that the composition of the left vertical arrows is $(\beta_X \ast \beta_Y)_A$, the vertical right arrow is $(\beta_{X\otimes Y})_A$, the horizontal upper arrow is the lax structure of $(RX \otimes -)$, and the horizontal lower arrow is the lax structure of $R(X \otimes G(-))$. Thus, the commutativity of the outer square encodes that $\beta$ is a monoidal transformation. To prove this commutativity, we note that the inner rectangle commutes by \cite{FLP2}*{Lemma 3.2} and that the inner triangle commutes by \cite{FLP2}*{Lemma 3.1}.
Since $\beta$ is an isomorphism, the claim follows.
\end{proof}

\begin{lemma}[$\oop$-dual version of \Cref{lemma:compare_R_Phi}]\label{lemma:dual_compare_R_Phi}
Let $L \dashv G$ be an opmonoidal adjunction such that the projection formula holds for $L$.
Then we obtain the following diagram in $\Catoplax$
\begin{center}
\begin{tikzpicture}[baseline=($(D) + 0.5*(d)$)]
      \coordinate (r) at (10,0);
      \coordinate (d) at (0,-1.5);
      
      \node (D) {$\cD$};
      \node (C2) at ($(D) + (r)$) {$\cC$};
      \node (Dp) at ($(D) + (d)$) {$\End_{\rMod{\cC}}( \cD^G )$};
      \node (Cp2) at ($(Dp) + (r)$) {$\End_{\rMod{\cC}}( \cC )$};
      
      \draw[->] (D) to node[above]{$L$} (C2);

      \draw[->] (Dp) to node[below]{$\Phi = \Hom( G, L )$} (Cp2);

      \draw[->] (D) to  node[left]{$\Emb$} (Dp);
      \draw[->] (C2) to node[right]{$\Emb$} (Cp2);

      \draw[-implies,double equal sign distance,shorten >=3em,shorten <=3em] (C2) to node[below]{$\beta$} (Dp);
\end{tikzpicture}
\end{center}
where $\beta$ is a $2$-isomorphism in $\Catoplax$ (i.e., an opmonoidal natural isomorphism) given by
\[
\beta_{X}: (LX \otimes -) \xrightarrow{ (\projr{X}{-}^{L, L \dashv G})^{-1} } L(X \otimes G(-))
\]
for $X \in \cD$.
\end{lemma}

We are now able to give a proof of \Cref{theorem:frobenius_functors_frobenius_monoidal}. 

\begin{proof}[Proof of \Cref{theorem:frobenius_functors_frobenius_monoidal}]
Since the right projection formula morphisms are mutual inverses, the natural isomorphisms $\beta$ of \Cref{lemma:compare_R_Phi} and \Cref{lemma:dual_compare_R_Phi} coincide. Thus, we have a natural isomorphism
\[
\big( \cD \xrightarrow{\Emb} \End_{\rMod{\cC}}( \cD^G ) \xrightarrow{\Phi} \End_{\rMod{\cC}}( \cC ) \big) \cong
\big( \cD \xrightarrow{ F } \cC \xrightarrow{\Emb} \End_{\rMod{\cC}}( \cC ) \big)
\]
which is both monoidal and opmonoidal.
The functor $\Emb$ is strong monoidal by \Cref{lemma:DvsEndrightDG} and $\Phi$ is Frobenius monoidal by \Cref{proposition:phi_frobenius_right_mod}.
Thus, the functor on the left-hand side is Frobenius monoidal as a composite of two Frobenius monoidal functors (\Cref{lemma:composition_of_Frobenius_monoidal_functors}).
It follows that the functor on the right-hand side $\Emb \circ F$ is also Frobenius monoidal by \Cref{lemma:transport_Frobenius_monoidal_structure}.
Last, since $\cC \xrightarrow{\Emb} \End_{\rMod{\cC}}( \cC )$ is a monoidal equivalence (\Cref{remark:C_equivalent_to_EndC}), it follows that $F$ is Frobenius monoidal by \Cref{lemma:Frobenius_along_equivalence}.
\end{proof}

Recall that for a monoidal category $\cC$, its $\otimes$-opposite $\cC^{\rev}$ is the monoidal category that has the same underlying category as $\cC$ and whose tensor product is given by 
$A \otimes_{\cC^{\rev}} B := B \otimes_{\cC} A.$
If we replace $\cC$ and $\cD$ by their $\otimes$-opposites, then we get the following reinterpretation of \Cref{theorem:frobenius_functors_frobenius_monoidal}:

\begin{corollary}[Left-sided version of \Cref{theorem:frobenius_functors_frobenius_monoidal}]
In \Cref{context:main}, assume that the left projection formula morphisms are mutual inverses. Then $(F, \lax^F, \oplax^F)$ is a Frobenius monoidal functor.
\end{corollary}

For a functor $G$ to be part of an ambiadjunction is not a property but a datum as there may be different non-isomorphic choices of ambiadjunctions $F\dashv G\dashv F$. Indeed, the setoid of ambiadjunctions for $G$ is equivalent to that of automorphisms of $F$, see \Cref{remark:ambiadjunctions_classified_by_aut}. Hence, results such as \Cref{theorem:frobenius_functors_frobenius_monoidal} on whether $F$ defines a Frobenius monoidal functor depend on the choice of datum of the ambiadjunction (within \Cref{context:main}).

We end this section by providing a result that can be used to describe the mentioned space of automorphism internal to the bicategory of right $\cC$-modules $\rMod{\cC}$ in order to control the ambiguity of the choice of isomorphism of left and right adjoints.

\begin{proposition}\label{proposition:EndF}
Let $G\dashv R$ be a monoidal adjunction such that the right projection formula holds for $R$.
Then we have an isomorphism of monoids
\begin{align*}
\End_{\cD}(\one) &\isomorph \End_{\Hom_{\rMod{\cC}}( \cD^G, \cC )}(R)\\
a &\longmapsto (RX \cong R(\one \otimes X) \xrightarrow{R(a\otimes X)} R(\one \otimes X) \cong RX)_{X \in \cD}
\end{align*}
where we regard $R$ as a morphism in $\rMod{\cC}$ by \Cref{lemma:module_adjunction_R}
\end{proposition}
\begin{proof}
We have a bijection
\begin{align*}
\End_{\Hom_{\rMod{\cC}}( \cD^G, \cC )}(R) &\isomorph \Hom_{\End_{\rMod{\cC}}( \cC )}( \id_\cC, RG ) \\
(R \xrightarrow{\alpha} R) &\longmapsto (\id_\cC \xrightarrow{\unit} RG \xrightarrow{\alpha \ast G} RG) \\
(R \xrightarrow{\beta \ast R} RGR \xrightarrow{R \ast \counit} R) &\longmapsfrom (\id_\cC \xrightarrow{\beta} RG)
\end{align*}
Moreover, by \Cref{remark:C_equivalent_to_EndC}, we have a bijection
\begin{align*}
\Hom_{\End_{\rMod{\cC}}( \cC )}( \id_\cC, RG ) &\isomorph \Hom_{\cC}( \id_\cC(\one), RG(\one) ) \\
\beta &\longmapsto \beta_{\one}
\end{align*}
Last, we have bijections
\[
\Hom_{\cC}( \id_\cC(\one), RG(\one) ) \cong \Hom_{\cD}( G(\one), G(\one))  \cong \End_{\cD}(\one).
\]
The resulting bijection is given as in the statement of this proposition, which clearly is a monoid morphism.
\end{proof}

\begin{remark}
    The above result \Cref{proposition:EndF} implies that the space of automorphisms of $R$ internal to $\rMod{\cC}$ is given by the automorphisms of $\one_\cD$. In particular, if $\cD$ is a tensor category in the sense of \cite{EGNO}, then the projection formula holds for $R$ as $\cD$ is rigid \cite{FLP2}*{Corollary~3.20} and this space of automorphisms is simply given by the units $\Bbbk^\times$ of the field $\Bbbk$. This is the case for the examples studied in \Cref{sec:Hopf}.
\end{remark}

\section{Frobenius monoidal functors on Drinfeld centers}\label{section:functors_on_centers}

The main result of this section, \Cref{thm:ZF}, shows how to induce Frobenius monoidal functors on Drinfeld centers. 
From the previous sections' results \Cref{theorem:relation_frobenius_monoidal_and_ambi} and \Cref{theorem:strong_monoidal_to_ambiadjunction}, we immediately get the following result  (analogous to the one-sided version \Cref{proposition:phi_frobenius_right_mod}).

\begin{proposition}\label{proposition:immediate_consequences}
In \Cref{context:main}, assume that the projection formula morphisms are mutual inverses. Then we get Frobenius monoidal functors
\[
\begin{tikzpicture}[mylabel/.style={fill=white}]
      \coordinate (r) at (8,0);
      \node (A) {$\End_{\BiMod{\cC}}(\cC)$};
      \node (B) at ($(A) + (r)$) {$\End_{\BiMod{\cC}}(\cD^G)$};
      \draw[->, out = 15, in = 180-15] (A) to node[mylabel]{$\Gamma = \Hom( F, G )$} (B);
      \draw[<-, out = -15, in = 180+15] (A) to node[mylabel]{$\Phi = \Hom( G, F )$} (B);
\end{tikzpicture} 
\]
and oplax-lax adjunctions $\Gamma \dashv \Phi$ and $\Phi \dashv \Gamma$.
\end{proposition}

The category $\End_{\BiMod{\cC}}(\cC)$ is a way to model the so-called Drinfeld center $\cZ(\cC)$ of $\cC$.
It is the goal of this section to prove that \Cref{proposition:immediate_consequences} yields a braided Frobenius monoidal functor from the Drinfeld center of $\cD$ to the Drinfeld center of $\cC$.

\subsection{Centers and endomorphism categories}
\label{sec:centers-endo}

For any $\cC$-bimodule $\cM$, there is the notion of its \emph{center} $\cZ(\cM)$, see \cite{GNN}*{Section~2B} and \cite{Gre}*{Section~2.3}.
In this paper, we will only make use of $\cZ(\cD^G)$.
Thus, for convenience, we state the definition of the center explicitly for $\cD^G$.

\begin{definition}\label{def::Z}
Let $G: \cC \rightarrow \cD$ be a strong monoidal functor.
The \emph{center} $\cZ(\cD^G)$ of the $\cC$-bimodule $\cD^G$ is the following category.
Its objects are given by pairs $(X,c^X)$, where $X\in\cD$ and $c^X$ is a \emph{half-braiding}, i.e., an isomorphism 
$$c^X_A \colon X\triangleleft A = X \otimes G(A) \isomorph G(A) \otimes X = A\triangleright X,$$
natural in $A \in \cC$, satisfying the coherences that the diagram 
\begin{equation}\label{eq:Ztensorcomp}
\vcenter{\hbox{\xymatrix{
X\otimes G(A\otimes B) \ar[rrr]^{c^X_{A\otimes B}}\ar[d]_{X \otimes \oplax^G_{A,B}}&&&G(A\otimes B) \otimes X\ar[d]^{\oplax^G_{A,B} \otimes X}\\
X \otimes G(A) \otimes G(B)\ar[d]_{c_A^X \otimes G(B)}&&&G(A) \otimes G(B) \otimes X\\
G(A) \otimes X \otimes G(B)\ar[rrr]^{\id}&&& G(A) \otimes X \otimes G(B) \ar[u]_{G(A) \otimes c^X_B}
}}}
\end{equation}
commutes for all $A,B\in\cC$ and that the diagram
\begin{equation}\label{eq:Zunitcomp}
    \xymatrix{
    X \otimes G(\one) \ar[drr]_{X \otimes \oplax^G_0\quad}\ar[rrrr]^{c^X_\one} &&&& G(\one) \otimes X,\ar[dll]^{\oplax^G_0 \otimes X} \\
    &&X&&
    }
\end{equation}
commutes.

A morphism $f\colon (X,c^X)\to(Y,c^Y)$ in $\cZ(\cD^G)$ is given by a morphism $f\in\Hom_\cD(X,Y)$ which commutes with the half-braidings, i.e., for all objects $A$ of $\cC$, the following diagram commutes:
\begin{align}\label{eq:Zmorcomp}
\xymatrix{
X\otimes G(A)\ar[d]^{f\otimes G(A)}\ar[rrr]^{c^X_A} &&& G(A) \otimes X\ar[d]^{G(A)\otimes f}\\
Y\otimes G(A)\ar[rrr]^{c^Y_A}  &&& G(A)\otimes Y.
}
\end{align}

Moreover, $\cZ(\cD^G)$ has a monoidal structure given by $(X,c^X)\otimes(Y,c^Y)=(X\otimes Y, c^{X\otimes Y})$, where
$$
c^{X\otimes Y}_A
:=
X\otimes Y \otimes G(A)
\xrightarrow{(X\otimes c^Y_A)} X\otimes G(A)\otimes Y
\xrightarrow{(c^X_A \otimes Y)} G(A)\otimes X\otimes Y
$$
for $A \in \cC$, $(X,c^X),(Y,c^Y) \in \cZ(\cD^G)$ (see {\cite{Maj2}*{Definition 4.2}}). The tensor unit of $\cZ(\cD^G)$ is given by the tensor unit $\one_{\cD}$ in $\cD$ with the half-braiding given by the identity:
\[
 \one_{\cD} \triangleleft A = \one_{\cD} \otimes G(A) = G(A) \otimes \one_{\cD} = A\triangleright \one_{\cD}.
\]
\end{definition}

\begin{definition}
Let $\cC$ be a strict monoidal category.
The center $\cZ(\cC):=\cZ(\cC^{\id_\cC})$ is called the \emph{Drinfeld center of $\cC$}.
\end{definition}

\begin{remark}
By \cite{JS}*{Proposition~4}, $\cZ(\cC)$ has a braiding given by 
$$
\Psi_{(A,c^A),(B,c^B)}
:= c^A_B \colon A\otimes B \isomorph B\otimes A
$$
for $A, B \in \cC$.
\end{remark}
 
Next, we compare the notion of a center with the endomorphism category of a bimodule category.

\begin{definition}\label{definition:center_prime} Let $\cM$ be a $\cC$-bimodule. Define
$$\cZ'(\cM):=\End_{\BiMod{\cC}}(\cM)$$
to be the category of $\cC$-bimodule endofunctors on $\cM$.
\end{definition}

\begin{remark}
Composition $KL=K\circ L$ makes $\cZ'(\cM)$ a strict monoidal category, see \Cref{definition:monoidal_end}.
\end{remark}

\begin{remark}\label{remark:ZprimeDG_explicitly2}
We spell out \Cref{definition:center_prime} explicitly in the case of $\cD^G$.
Objects in $\cZ'(\cD^G)$ are endofunctors 
\[
K: \cD \rightarrow \cD
\]
equipped with isomorphisms (left and right lineators)
\begin{align*}
\llineator_{A,X}: K(G(A) \otimes X) \isomorph G(A) \otimes K(X) & & \rlineator_{X,A}: K(X \otimes G(A)) \isomorph K(X) \otimes G(A)
\end{align*}
natural in $A \in \cC$, $X \in \cD$ such that \Cref{equation:K_rlinor_coh}, \Cref{equation:K_rlinor_one_coh} and
the following diagrams commute for $A, B \in \cC$, $X \in \cD$:

\begin{equation}\label{equation:K_llinor_coh}
   \begin{tikzpicture}[baseline=($(21)$)]
  \coordinate (r) at (4,0);
  \coordinate (d) at (0,-2.5);
  \node (11) {$K( G(A\otimes B) \otimes X)$};
  \node (21) at ($(11) +(d) - (r)$) {$K( GA \otimes GB \otimes X)$};
  \node (22) at ($(11) + (d) + (r)$) {$G(A \otimes B) \otimes K(X)$};
  \node (31) at ($(21) + (d)$) {$GA \otimes K(GB\otimes X)$};
  \node (32) at ($(22) + (d)$) {$GA\otimes GB\otimes K(X)$};
  \draw[->] (11) to node[above left]{$K(\oplax^G_{A,B} \otimes X)$} (21);
  \draw[->] (11) to node[above right]{$\llineator_{A\otimes B,X}$} (22);
  \draw[->] (21) to node[left]{$\llineator_{A,GB\otimes X}$} (31);
  \draw[->] (22) to node[right]{$\oplax^G_{A,B} \otimes K(X)$} (32);
  \draw[->] (31) to node[below]{$GA\otimes \llineator_{B,X}$} (32);
\end{tikzpicture} 
\end{equation}

\begin{equation}\label{equation:K_llinor_one_coh}
    \begin{tikzpicture}[baseline=($(11) + 0.5*(d)$)]
      \coordinate (r) at (4,0);
      \coordinate (d) at (0,-2);
      \node (11) {$K( G\one \otimes X)$};
      \node (12) at ($(11) + 2*(r)$) {$G\one \otimes K(X)$};
      \node (2) at ($(11) + (d) + (r)$) {$K(X)$};
      \draw[->] (11) to node[above]{$\llineator_{\one,X}$} (12);
      \draw[->] (11) to node[left, xshift=-0.5em]{$K(\oplax^G_{\one} \otimes X)$} (2);
      \draw[->] (12) to node[right]{$\oplax^G_{\one} \otimes K(X)$} (2);
    \end{tikzpicture} 
\end{equation}

\begin{equation}\label{equation:K_llinor_rlinor_coh}
  \begin{tikzpicture}[baseline=(21)]
    \coordinate (r) at (6,0);
    \coordinate (d) at (0,-2);
    
    \node (11) {$K( GA \otimes X \otimes GB)$};
    \node (12) at ($(11)+(r)$) {$K(GA \otimes X \otimes GB)$};

    \node (21) at ($(11) + (d) - 0.25*(r)$){$GA \otimes K(X \otimes GB)$};
    \node (22) at ($(12)+ (d) + 0.25*(r)$) {$K(GA \otimes X) \otimes GB$};

    \node (31) at ($(11)+2*(d)$) {$GA \otimes KX \otimes GB$};
    \node (32) at ($(12)+2*(d)$) {$GA \otimes KX \otimes GB$};

    \draw[<-,thick] (31) to node[left]{$GA \otimes \rlineator_{X,B}$}(21);
    \draw[<-,thick] (21) to node[left]{$\llineator_{A,X\otimes GB}$}(11);
    \draw[<-,thick] (12) to node[above]{$\id$}(11);

    \draw[<-,thick] (32) to node[below]{$\id$}(31);
    \draw[<-,thick] (32) to node[right]{$\llineator_{A,X} \otimes GB$}(22);
    \draw[<-,thick] (22) to node[right,xshift=0.2em]{$\rlineator_{GA \otimes X,B}$}(12);
  \end{tikzpicture}
\end{equation}

A morphism from $(K, \llineator^K, \rlineator^K)$ to $(M, \llineator^M, \rlineator^M)$ in $\cZ'(\cD^G)$ is given by a natural transformation $\tau: K \rightarrow M$ such that \Cref{equation:tau_rlin_coh} and the following diagram commute for $A \in \cC$, $X \in \cD$:
\begin{equation}\label{equation:tau_llin_coh}
\begin{tikzpicture}[ baseline=($(11) + 0.5*(d)$)]
      \coordinate (r) at (7,0);
      \coordinate (d) at (0,-2);
      \node (11) {$K( GA\otimes X)$};
      \node (12) at ($(11) + (r)$) {$GA\otimes K(X)$};
      \node (21) at ($(11) + (d)$) {$M(GA\otimes  X)$};
      \node (22) at ($(11) + (d) + (r)$) {$GA\otimes M(X)$};
      \draw[->] (11) to node[above]{$\llineator^K_{A,X}$} (12);
      \draw[->] (21) to node[below]{$\llineator^M_{A,X}$} (22);
      \draw[->] (11) to node[left]{$\tau_{GA\otimes X}$} (21);
      \draw[->] (12) to node[right]{$GA\otimes \tau_{X}$} (22);
\end{tikzpicture}
\end{equation}

\end{remark}

\begin{example}
    If $\cC$ is the trivial monoidal category with only one object, then any category $\cM$ is a $\cC$-bimodule category in a canonical way and $\cZ'(\cM)$ is simply the monoidal category of all endofunctors on $\cM$. 
\end{example}

Next, we compare the monoidal categories $\cZ'(\cD^G)$ and $\cZ( \cD^G )$.

\begin{lemma}\label{prop::Z-vs-Z-prime} 
Let $G\colon \cC\to \cD$ be a strong monoidal functor. Then we have a strong monoidal functor
\begin{align*}
\Emb\colon \cZ(\cD^G) &\to \cZ'(\cD^G) \\
(X,c^X) &\mapsto (X\otimes(-),\llineator^{(X,c^X)},\rlineator^{(X,c^X)})
\end{align*}
where
\[
\llineator^{(X,c^X)}_{A,Y}: X \otimes GA \otimes Y \xrightarrow{c^X_A \otimes Y} GA \otimes X \otimes Y
\]
and
\[
\rlineator^{(X,c^X)}_{Y,A}: X \otimes Y \otimes GA \xrightarrow{\id} X \otimes Y \otimes GA
\]
for $A \in \cC$, $Y \in \cD$. The underlying natural isomorphism of
\[
\lax^{\Emb}: \Emb( (X,c^X) )\Emb( (Y,c^Y) ) \rightarrow \Emb( (X,c^X) \otimes (Y,c^Y) )
\]
is given by the identity of $(X \otimes Y \otimes -): \cD \rightarrow \cD$ for $(X,c^X), (Y,c^Y) \in \cZ( \cD^G )$, and the underlying natural isomorphism of
\[
\lax_0^{\Emb}: \id_{\cD^G} \rightarrow \Emb( \one_{\cZ( \cD^G )} )
\]
is given by the identity of $\id_{\cD^G}$.
\end{lemma}
\begin{proof}
First, we check that $(X\otimes(-),\llineator^{(X,c^X)},\rlineator^{(X,c^X)})$ satisfies the coherences of \Cref{remark:ZprimeDG_explicitly2} for objects:
the coherence \eqref{equation:K_llinor_coh} follows directly from \Cref{eq:Ztensorcomp}.
The coherence \eqref{equation:K_llinor_one_coh} follows directly from \Cref{eq:Zunitcomp}.
The coherences \eqref{equation:K_rlinor_coh}, \eqref{equation:K_rlinor_one_coh}, and \eqref{equation:K_llinor_rlinor_coh} degenerate to simple equalities of morphisms and hold trivially.

Next, let $f\colon (X,c^X)\to(Y,c^Y)$ be a morphism in $\cZ(\cD^G)$.
We need to check that $(f \otimes -): (X \otimes -) \rightarrow (Y \otimes -)$ satisfies the coherences of \Cref{remark:ZprimeDG_explicitly2} for morphisms:
the coherence \eqref{equation:tau_llin_coh} follows directly from \Cref{eq:Zmorcomp}.
The coherence \eqref{equation:tau_rlin_coh} degenerates to a simple equality of morphisms and holds trivially.

Last, we need to check that the identity of $(X \otimes Y \otimes -)$ turns $\Emb$ into a strong monoidal functor.
The right lineators of $\Emb( (X,c^X) )\Emb( (Y,c^Y) )$ and 
$\Emb( (X,c^X) \otimes (Y,c^Y) )$ are both given by identity morphisms.
The left lineator of $\Emb( (X,c^X) )\Emb( (Y,c^Y) )$ is given by
\[
X \otimes Y \otimes GA \otimes Z \xrightarrow{X \otimes c^Y_A \otimes Z} X \otimes GA \otimes Y \otimes Z \xrightarrow{c^X_A \otimes Y \otimes Z} GA \otimes X \otimes Y \otimes Z
\]
for $A \in \cC$, $Z \in \cD$.
The left lineator of $\Emb( (X,c^X) \otimes (Y,c^Y) ) = \Emb( (X \otimes Y, c^{X \otimes Y} ) )$ is given by $c^{X \otimes Y}_A \otimes Z$.
Thus, we really have an equality $\Emb( (X,c^X) )\Emb( (Y,c^Y) ) = \Emb( (X,c^X) \otimes (Y,c^Y) )$ not only on the level of functors, but on the level of objects in $\cZ'(\cD^G)$, and this equality is natural in $(X,c^X), (Y,c^Y) \in \cZ( \cD^G )$.
Moreover, we have  $\id_{\cD^G} = \Emb( \one_{\cZ( \cD^G )} )$ as objects in $\cZ'(\cD^G)$. Thus, the coherences of $\Emb$ for being a strong monoidal functor hold trivially.
\end{proof}

\begin{remark}\label{remark:equivalence_Z_Zp}
It is well-known, see e.g.\ \cite{EGNO}*{Proposition 7.13.8.}, that if $G$ is the identity functor of $\cC$, then \Cref{prop::Z-vs-Z-prime} gives us a monoidal equivalence
\[
\cZ( \cC ) \simeq \cZ'( \cC^{\id_\cC} ).
\]
\end{remark}

\subsection{Induced functors on centers}

We require the following two results from \cite{FLP2}.

\begin{lemma}\label{lemma:ZRDG_is_lax}
Suppose given a monoidal adjunction $G \dashv R$ such that the projection formula holds for $R$.
Then we obtain a lax monoidal functor
\begin{align*}
\cZ( \cD^G ) &\xrightarrow{\cZ( R )} \cZ( \cC ) \\
(X, c^X) &\longmapsto (RX, c^{RX})
\end{align*}
where
\[
c_A^{RX}=\big(RX\otimes A\xrightarrow{\projr{X}{A}^{R, G \dashv R}}R(X\otimes GA)\xrightarrow{R(c^X_{A})}R(GA\otimes X)\xrightarrow{(\projl{A}{X}^{R, G \dashv R})^{-1}}A\otimes RX \big)
\]
for $A\in \cC$ and where the lax structure is directly inherited from the lax structure of $R$, i.e., it is given by
\begin{equation}
\lax^{\cZ(R)}_{ (X,c^X), (Y,c^Y) } = \lax^R_{X,Y} \qquad \qquad\lax^{\cZ(R)}_0 = \lax^R_0    
\end{equation}
for $(X,c^X), (Y,c^Y) \in \cZ( \cD^G )$.
\end{lemma}
\begin{proof}
Follows from \cite{FLP2}*{Proposition 4.8}.
\end{proof}

\begin{lemma}[$\oop$-dual version of \Cref{lemma:ZRDG_is_lax}]\label{lemma:ZLDG_is_oplax}
Suppose given an opmonoidal adjunction $L \dashv G$ such that the projection formula holds for $L$.
Then we obtain an oplax monoidal functor
\begin{align*}
\cZ( \cD^G ) &\xrightarrow{\cZ( L )} \cZ( \cC ) \\
(X, c^X) &\longmapsto (LX, c^{LX})
\end{align*}
where
\[
c_A^{LX}=\big(LX\otimes A\xrightarrow{ (\projr{X}{A}^{L, L \dashv G})^{-1} }L(X\otimes GA)\xrightarrow{L(c^X_{A})}L(GA\otimes X)\xrightarrow{(\projl{A}{X}^{L, L \dashv G})}A\otimes LX \big)
\]
for $A\in \cC$ and where the oplax structure is directly inherited from the oplax structure of $L$, i.e., it is given by
\begin{equation}
\oplax^{\cZ(L)}_{ (X,c^X), (Y,c^Y) } = \oplax^L_{X,Y} \qquad \qquad\oplax^{\cZ(L)}_0 = \oplax^L_0    
\end{equation}
for $(X,c^X), (Y,c^Y) \in \cZ( \cD^G )$.
\end{lemma}

The next lemma compares the functor $\cZ(R)$ from \Cref{lemma:ZRDG_is_lax} with the functor $\Phi$ from \Cref{proposition:immediate_consequences}.

\begin{lemma}\label{lemma:compare_ZR_Phi_lax}
Let $G \dashv R$ be a monoidal adjunction such that the projection formula holds for $R$.

Then we obtain the following diagram in $\Catlax$
\begin{center}
\begin{tikzpicture}[baseline=($(D) + 0.5*(d)$)]
      \coordinate (r) at (10,0);
      \coordinate (d) at (0,-1.5);
      
      \node (D) {$\cZ(\cD^G)$};
      \node (C2) at ($(D) + (r)$) {$\cZ(\cC)$};
      \node (Dp) at ($(D) + (d)$) {$\cZ'( \cD^G )$};
      \node (Cp2) at ($(Dp) + (r)$) {$\cZ'( \cC )$};
      
      \draw[->] (D) to node[above]{$\cZ( R )$} (C2);

      \draw[->] (Dp) to node[below]{$\Phi = \Hom( G, R )$} (Cp2);

      \draw[->] (D) to  node[left]{$\Emb$} (Dp);
      \draw[->] (C2) to node[right]{$\Emb$} (Cp2);

      \draw[-implies,double equal sign distance,shorten >=3em,shorten <=3em] (C2) to node[below]{$\beta$} (Dp);
\end{tikzpicture}
\end{center}
where $\beta$ is a $2$-isomorphism in $\Catlax$ (i.e., a monoidal natural isomorphism) given by
\[
\beta_{(X,c^X)}: (RX \otimes -) \xrightarrow{\projr{X}{-}^{R, G \dashv R}} R(X \otimes G(-))
\]
for $(X, c^X) \in \cZ(\cD^G)$.

\end{lemma}
\begin{proof}
The two embedding functors are strong monoidal by \Cref{prop::Z-vs-Z-prime}.
The functor $\Phi$ is lax by \Cref{proposition:induced_end_adjunctions}.
The functor $\cZ(R)$ is lax by \Cref{lemma:ZRDG_is_lax}.
Thus, all functors lie in $\Catlax$.

We need to show that the components of $\beta$ lie in $\cZ'( \cC )$. For this, let $(X, c^X) \in \cZ(\cD^G)$. Then we have the following diagram:
\begin{center}
  \begin{tikzpicture}[every node/.style={scale=0.8}]
    \coordinate (r) at (4.1,0);
    \coordinate (d) at (0,-2);
    
    \node (11) {$RX \otimes A \otimes B$};
    \node (12) at ($(11)+(r)$) {$R(X \otimes GA) \otimes B$};
    \node (13) at ($(12) + (r)$) {$R(GA \otimes X) \otimes B$};
    \node (14) at ($(13) + (r)$) {$A \otimes RX \otimes B$};

    \node (21) at ($(11) + (d)$){$R(X \otimes G(A \otimes B) )$};
    \node (22) at ($(21)+(r)$) {$R( X \otimes GA \otimes GB )$};
    \node (23) at ($(22) + (r)$) {$R( GA \otimes X \otimes GB )$};
    \node (24) at ($(23) + (r)$) {$A \otimes R( X \otimes GB )$};
    
    \draw[->,thick] (11) to node[above]{$\projr{X}{A} \otimes B$}(12);
    \draw[->,thick] (12) to node[above]{$R( c^X_A) \otimes B$} (13);
    \draw[->,thick] (13) to node[above]{$\projl{A}{X}^{-1} \otimes B$} (14);

    \draw[->,thick] (21) to node[below]{$R( X \otimes \oplax^G_{A,B} )$}(22);
    \draw[->,thick] (22) to node[below]{$R( c^X_A \otimes GB )$}(23);
    \draw[->,thick] (23) to node[below]{$\projl{A}{X \otimes GB}^{-1}$}(24);

    \draw[->,thick] (11) to node[left]{$\projr{X}{A \otimes B}$} (21);
    \draw[->,thick] (12) to node[left]{$\projr{X \otimes GA}{B}$}(22);
    \draw[->,thick] (13) to node[left]{$\projr{GA \otimes X}{B}$}(23);
    \draw[->,thick] (14) to node[left]{$A \otimes \projr{X}{B}$}(24);
  \end{tikzpicture}
\end{center}
The composition of the morphisms in the top row give the left lineator of the source of $\beta$, i.e., of $\Emb \cZ( R )( ( X, c^X ) )$.
The composition of the morphisms in the bottom row give the left lineator of the range of $\beta$, i.e., of $\Phi\Emb( ( X, c^X ) )$.
The commutativity of the outer rectangle thus encodes the compatibility of $\beta$ with the left lineators.
We show the commutativity of all three small rectangles: the left rectangle commutes by \cite{FLP2}*{Lemma 3.4}, the middle rectangle commutes by naturality of the projection formula morphism, and the right rectangle commutes by \cite{FLP2}*{Lemma 3.5}.

Furthermore, by \Cref{remark:rproj_as_a_module_morphism},
$\beta$ is compatible with the right lineators.
Hence, the components of $\beta$ lie in $\cZ'( \cC )$.

Next, we need to show that $\beta$ is a monoidal transformation. By definition, this means that the following diagram commutes for all $(X,c^X), (Y,c^Y) \in \cZ( \cD^G )$:
\begin{center}
  \begin{tikzpicture}[mylabel/.style={fill=white}]
    \coordinate (r) at (9,0);
    \coordinate (d) at (0,-2);
    
    \node (11) {$\Emb( \cZ(R)( (X,c^X) ) ) \circ \Emb( \cZ(R)( (Y,c^Y) ) )$};
    \node (12) at ($(11)+(r)$) {$\Emb( \cZ(R)( (X,c^X) \otimes (Y, c^Y) ) )$};

    \node (21) at ($(11) + (d)$){$\Phi( \Emb( (X,c^X) ) ) \circ \Phi( \Emb( (Y, c^Y) ) )$};
    \node (22) at ($(21)+(r)$) {$\Phi( \Emb( (X,c^X) \otimes (Y,c^Y) ) )$};
    
    \draw[->,thick] (11) to node[above]{$\lax^{\Emb \circ \cZ(R)}_{(X,c^X),(Y,c^Y)}$}(12);

    \draw[->,thick] (21) to node[above]{$\lax^{\Phi \circ \Emb}_{(X,c^X),(Y,c^Y)}$}(22);

    \draw[->,thick] (11) to node[left]{$\beta_{(X,c^X)} \ast \beta_{(Y,c^Y)}$}(21);
    \draw[->,thick] (12) to node[right]{$\beta_{(X,c^X) \otimes (Y,c^Y)}$}(22);
  \end{tikzpicture}
\end{center}
On the level of functors and natural transformations, this square is given by the outer square of the diagram in \Cref{equation:beta_monoidal}.

Last, since the projection formula holds, $\beta$ is also an isomorphism.
The claim follows.
\end{proof}

\begin{lemma}[$\oop$-dual version of \Cref{lemma:compare_ZR_Phi_lax}]\label{lemma:compare_ZL_Phi_oplax}
Let $L \dashv G$ be an opmonoidal adjunction such that the projection formula holds for $L$.
Then we obtain the following diagram in $\Catoplax$
\begin{center}
\begin{tikzpicture}[baseline=($(D) + 0.5*(d)$)]
      \coordinate (r) at (10,0);
      \coordinate (d) at (0,-1.5);
      
      \node (D) {$\cZ(\cD^G)$};
      \node (C2) at ($(D) + (r)$) {$\cZ(\cC)$};
      \node (Dp) at ($(D) + (d)$) {$\cZ'( \cD^G )$};
      \node (Cp2) at ($(Dp) + (r)$) {$\cZ'( \cC )$};
      
      \draw[->] (D) to node[above]{$\cZ( L )$} (C2);

      \draw[->] (Dp) to node[below]{$\Phi = \Hom( G, L )$} (Cp2);

      \draw[->] (D) to  node[left]{$\Emb$} (Dp);
      \draw[->] (C2) to node[right]{$\Emb$} (Cp2);

      \draw[-implies,double equal sign distance,shorten >=3em,shorten <=3em] (C2) to node[below]{$\beta$} (Dp);
\end{tikzpicture}
\end{center}
where $\beta$ is a $2$-isomorphism in $\Catoplax$ (i.e., an opmonoidal natural isomorphism) given by
\[
\beta_{(X,c^X)}: (LX \otimes -) \xrightarrow{ (\projr{X}{-}^{L, L \dashv G})^{-1} } L(X \otimes G(-))
\]
for $(X, c^X) \in \cZ(\cD^G)$. \qedsymbol
\end{lemma}

\begin{proposition}\label{proposition:ZF_on_modules}
In \Cref{context:main}, assume that the projection formula morphisms are mutual inverses.
Then we obtain a Frobenius monoidal functor
\begin{align*}
\cZ( \cD^G ) &\xrightarrow{\cZ( F )} \cZ( \cC ) \\
(X, c^X) &\longmapsto (FX, c^{FX})
\end{align*}
where
\[
c_A^{FX}=\big(FX\otimes A\xrightarrow{\projr{X}{A}^{F, G \dashv F}}F(X\otimes GA)\xrightarrow{F(c^X_{A})}F(GA\otimes X)\xrightarrow{(\projl{A}{X}^{F, G \dashv F})^{-1}}A\otimes FX \big)
\]
for $A\in \cC$ and where the lax and oplax structure is directly inherited from the lax and oplax structure of $F$, i.e., it is given by
\begin{equation}
\lax^{\cZ(F)}_{ (X,c^X), (Y,c^Y) } = \lax^F_{X,Y} \qquad \qquad\lax^{\cZ(F)}_0 = \lax^F_0    
\end{equation}
\begin{equation}
\oplax^{\cZ(F)}_{ (X,c^X), (Y,c^Y) } = \oplax^F_{X,Y} \qquad \qquad\oplax^{\cZ(F)}_0 = \oplax^F_0    
\end{equation}
for $(X,c^X), (Y,c^Y) \in \cZ( \cD^G )$.
\end{proposition}
\begin{proof}
We set $L := F$ and $R := F$. Note that the arguments are very similar to the arguments in the proof of \Cref{theorem:frobenius_functors_frobenius_monoidal}.

Our assumptions imply that $\cZ( R )$ (obtained from \Cref{lemma:ZRDG_is_lax}) and $\cZ( L )$ (obtained from \Cref{lemma:ZLDG_is_oplax}) are equal as functors, and both given by the functor $\cZ( F )$ as defined in the statement.
It follows that $\cZ( F )$ has the lax and oplax structures as defined in the statement.

Moreover, our assumptions imply that we can apply both \Cref{lemma:compare_ZR_Phi_lax} and \Cref{lemma:compare_ZL_Phi_oplax}.
Using the notation of \Cref{lemma:compare_ZR_Phi_lax}, we obtain a natural isomorphism
\[
\big( \cZ( \cD^G ) \xrightarrow{\Emb} \cZ'( \cD^G ) \xrightarrow{\Phi} \cZ'( \cC ) \big) \cong
\big( \cZ( \cD^G ) \xrightarrow{ \cZ( F ) } \cZ( \cC ) \xrightarrow{\Emb} \cZ'( \cC ) \big)
\]
which is both monoidal and opmonoidal.
The functor $\Emb$ is strong monoidal by \Cref{prop::Z-vs-Z-prime} and $\Phi$ is Frobenius monoidal by \Cref{proposition:immediate_consequences}.
Thus, the functor on the left-hand side is Frobenius monoidal as a composite of two Frobenius monoidal functors (\Cref{lemma:composition_of_Frobenius_monoidal_functors}).
It follows that the functor on the right-hand side $\Emb \circ \cZ( F )$ is also Frobenius monoidal by \Cref{lemma:transport_Frobenius_monoidal_structure}.
Last, since $\Emb: \cZ( \cC) \rightarrow \cZ'( \cC )$ is a monoidal equivalence (\Cref{remark:equivalence_Z_Zp}), it follows that $\cZ( F )$ is Frobenius monoidal by \Cref{lemma:Frobenius_along_equivalence}.
\end{proof}

We are now ready to derive the main result of this section giving Frobenius monoidal functors on Drinfeld centers which, by slight abuse of notation, we also denote by $\cZ(F)$.

\begin{theorem}[Main induction theorem]\label{thm:ZF}
In \Cref{context:main}, assume that the projection formula morphisms are mutual inverses.
Then we obtain a braided Frobenius monoidal functor
\begin{align*}
\cZ( \cD ) &\xrightarrow{\cZ( F )} \cZ( \cC ) \\
(X, c^X) &\longmapsto (FX, c^{FX})
\end{align*}
where
\[
c_A^{FX}=\big(FX\otimes A\xrightarrow{\projr{X}{A}^{F, G \dashv F}}F(X\otimes GA)\xrightarrow{F(c^X_{GA})}F(GA\otimes X)\xrightarrow{(\projl{A}{X}^{F, G \dashv F})^{-1}}A\otimes FX \big)
\]
for $A\in \cC$ and where the lax and oplax structure is directly inherited from the lax and oplax structure of $F$, i.e., it is given by
\begin{equation}
\lax^{\cZ(F)}_{ (X,c^X), (Y,c^Y) } = \lax^F_{X,Y} \qquad \qquad\lax^{\cZ(F)}_0 = \lax^F_0    
\end{equation}
\begin{equation}
\oplax^{\cZ(F)}_{ (X,c^X), (Y,c^Y) } = \oplax^F_{X,Y} \qquad \qquad\oplax^{\cZ(F)}_0 = \oplax^F_0    
\end{equation}
for $(X,c^X), (Y,c^Y) \in \cZ( \cD^G )$.

\end{theorem}
\begin{proof} 
We compose the Frobenius monoidal functor of \Cref{proposition:ZF_on_modules} with the strong monoidal functor $\cZ(\cD)\to \cZ(\cD^G)$ from \cite{FLP2}*{Lemma 4.7}:
\[
\cZ(\cD) \to \cZ(\cD^G) \to \cZ( \cC ).
\]
It follows that the resulting functor is Frobenius monoidal.
Moreover, the resulting functor is lax braided and oplax braided by \cite{FLP2}*{Theorem 4.10 and 4.15}. Thus, it is braided.

\end{proof}

In the case when the strong monoidal functor $G$ is braided, it suffices for the right (or left) projection formula morphisms to be mutual inverses to obtain a Frobenius monoidal functor on Drinfeld centers. We state this result in the case when the right projection formulas are mutual inverses.

\begin{corollary}\label{cor:braided-case}
   In \Cref{context:main}, assume that the right projection formula morphisms are mutual inverses.  Additionally, assume that $(\cC,\Psi^\cC)$ and $(\cD,\Psi^\cD)$ are braided such that $G\colon \cC\to \cD$ is a braided functor. Then we obtain a braided Frobenius monoidal functor $$\cZ(F)\colon \cZ(\cD)\to \cZ(\cC),\quad (X,c^X)\mapsto (FX,c^{FX}), $$
   where $c_A^{FX}$ is defined by the diagram
\[
\xymatrix{
\ar[d]_{c_A^{FX}}FX\otimes A\ar[rr]^{\projr{X}{A}^{F, G \dashv F}}&&F(X\otimes GA)\ar[rr]^{c^X_{GA}}&&\ar[d]^{F(\Psi^{\cD}_{X,GA})^{-1}}F(GA\otimes X)\\
A\otimes FX&&\ar[ll]^{\Psi^{\cD}_{FX,A}}FX\otimes A &&F(X\otimes GA)\ar[ll]^{(\projr{X}{A}^{F, G \dashv F})^{-1}}
}
\]
and where the lax and oplax structure is directly inherited from the lax and oplax structure of $F$ as in \Cref{thm:ZF}.
\end{corollary}
\begin{proof}
By \Cref{lemma:braided_context}, the projection formula morphisms are mutual inverses and we can apply \Cref{thm:ZF}. The explicit formula for the half-braiding $c^{FX}$ is obtained using \cite{FLP2}*{Lemma~3.13}.
\end{proof}

\begin{remark}
Our results in  \cite{FLP2}*{Theorems~4.10 and 4.15} imply that in \Cref{context:main}, there are two induced functors on Drinfeld centers
$$\cZ(R),\cZ(L)\colon \cZ(\cD)\to \cZ(\cC),$$
where we write $F=R$ when $F$ is viewed as the right adjoint of $G$ and $F=L$ when $F$ is viewed as the left adjoint of $G$. The functor $\cZ(R)$ is braided lax monoidal while $\cZ(L)$ is braided oplax monoidal. Given that the underlying functors $R=L=F$ are equal, we have that 
$$\cZ(R)(X,c^X)=(FX,c^{RX}), \qquad \cZ(L)(X,c^X)=(FX,c^{LX}),$$
meaning that the underlying objects of the images, which are objects in $\cZ(\cC)$, coincide and the difference lays in the half-braidings. Moreover, the functors $\cZ(R)$ and $\cZ(L)$ are identical on morphisms.

Now assume that the right projection formula morphisms are mutual inverses but the left projection formula morphisms are \emph{not} (or vice versa). That is, the assumptions required for \Cref{context:main} holds but the stronger assumptions for \Cref{thm:ZF} do not hold. We will later present examples where this is the case (see \Cref{ex:uqsl2-cont}). Then the two functors $\cZ(R)$ and $\cZ(L)$ will not coincide. Indeed, by construction of the half-braidings of $\cZ(R)$ and $\cZ(L)$ one can show that if $\projrnoarg^{F,G\dashv F}=(\projrnoarg^{F,F\dashv G})^{-1}$, then, by invertibility of $F(c^X_A)$, we have
that $c^{RX}=c^{LX}$ for an object $(X,c^X)$ of $\cZ(\cD)$ implies that 
$(\projl{A}{X}^{F,F\dashv G})^{-1}=\iprojl{A}{X}$ for all $A\in \cC$.
In particular, if the forgetful functor $\cZ(\cD)\to \cD$ is essentially surjective, then, under the assumption that the right projection formula morphisms are mutual inverses, $\cZ(R)=\cZ(L)$ is equivalent to the left projection formula morphisms also being mutual inverses. 
\end{remark}

\section{Frobenius monoidal functors induced by morphisms of Hopf algebras}
\label{sec:Hopf}

In this section, we will investigate our general categorical results in the case when the strong monoidal functor $G$ is a restriction functor 
$$\Res_\varphi\colon\lMod{H}\to \lMod{K}$$
along a morphism of Hopf algebras $\varphi\colon K\to H$. It is known that $G$ is part of an ambiadjunction $F\dashv G\dashv F$, i.e., induction $\Ind_\varphi$ and coinduction $\CoInd_\varphi$ are naturally isomorphic, if and only if $H$ is finitely generated projective over $K$ and there exists a so-called \emph{Frobenius morphism} $\tr\colon H\to K$, see \Cref{thm:ind-coind}. In this case, we will:
\begin{enumerate}
    \item[(a)] Use \Cref{theorem:frobenius_functors_frobenius_monoidal} to show that $\Ind_\varphi\colon \lMod{K}\to \lMod{H}$ is a Frobenius monoidal functor if $\tr$ is a morphism of right $H$-comodules, see \Cref{cor:Ind-Frob}; 
    \item[(b)] Use \Cref{thm:ZF} to show that $\cZ(\Ind_\varphi)\colon \cZ(\lMod{K})\to \cZ(\lMod{H})$ is a braided Frobenius monoidal functor if $\tr$ is a morphism of right $H$-$H$-bicomodules, see \Cref{cor:ZInd-Frob}.
\end{enumerate}
Moreover, we will give an explicit example for which $\Ind_\varphi$ satisfies (a) but not (b).

\subsection{Induction and coinduction of modules over Hopf algebras}

For a Hopf algebra $H$ over a field $\Bbbk$ with coproduct 
$$\Delta=\Delta_H\colon H\to H\otimes H,\quad h\mapsto h_{(1)}\otimes h_{(2)}$$
(using Sweedler's notation), counit $\varepsilon=\varepsilon_H\colon H\to\Bbbk$, 
and invertible antipode $S=S_H\colon H\to H$, consider the monoidal category $\lMod{H}$ of $H$-modules (see e.g.~\cites{Kas,Maj,Rad} for details). Given $\varphi\colon K\to H$ a morphism of Hopf algebras, the restriction functor
$$
G=\Res_\varphi\colon \lMod{H}\to \lMod{K}, \quad W\mapsto \left.W\right|_K
$$
has the structure of a strong monoidal functor. Its left and right adjoints are given by 
\begin{align}
  L=\Ind_\varphi(V):=H\otimes_K V, \quad R=\CoInd_\varphi(V)=\Hom_{\lMod{K}}(H,V)=\Hom_{K}(H,V).
\end{align}
The left $H$-module structures are given, for $\Ind_\varphi$, by left multiplication and, for $\CoInd_\varphi$, by
\begin{align}
    h\cdot f(g)=f(gh), \qquad f\in \Hom_{K}(H,V), ~~ h,g\in H.
\end{align}
Moreover, $\Hom_{K}(H,K)$ is an $H$-$K$-bimodule with a right $K$-action given by the right action on the target $K$.

The unit and counit isomorphisms of the adjunctions $$L=\Ind_\varphi\dashv G=\Res_\varphi, \qquad G\dashv R=\CoInd_\varphi$$
are given by
\begin{align*}
    \unit^{L\dashv G}_V&\colon V\to GL(V), &v&\mapsto 1\otimes_K v,\\
    \counit^{L\dashv G}&\colon LG(W)\to W, &h\otimes_K w&\mapsto hw,\\
    \unit^{G\dashv R}_W&\colon W\longrightarrow RG(W)=\Hom_K(H,\left.W\right|_K),     &  w&\mapsto (h\mapsto hw),\\
    \counit^{G\dashv R}_V&\colon GR(V)=\left.\Hom_K(H,V)\right|_K\longrightarrow V, &  f&\mapsto f(1),
\end{align*}
for $V\in \lMod{K}$, $W\in \lMod{H}$.
The oplax and lax monoidal structures are given by 
\begin{gather}
\oplax^{L}_{V,W}\colon  L(V\otimes U)\to L(V)\otimes L(U), \quad h\otimes_K (u\otimes v)\mapsto (h_{(1)}\otimes_K v)\otimes (h_{(2)}\otimes_K u)\label{eq:oplax-Hopf},\\
    \lax^R_{V,W}\colon R(V)\otimes R(U)\longrightarrow R(V\otimes U), \quad f\otimes g\mapsto \big(h\mapsto f(h_{(1)})\otimes g(h_{(2)})\big),\label{eq:lax-Hopf}
\end{gather}
for $V,U\in \lMod{K}$.

\begin{lemma}[{\cite{FLP2}*{Section~7}}]\label{lem:proj-iso-L-Hopf}
The projection formulas hold for $\Ind_\varphi$ with the isomorphisms given by
\begin{align*}
    \iprojl{W}{V}&\colon L(G(W)\otimes V)\to W\otimes L(V), &h\otimes_K (w\otimes v)&\mapsto h_{(1)}w\otimes (h_{(2)}\otimes_K v),\\
        (\iprojl{W}{V})^{-1}&\colon W\otimes L(V)\to L(G(W)\otimes V), &w\otimes (h\otimes_K v)&\mapsto h_{(2)}\otimes_K (S^{-1}(h_{(1)})w\otimes  v),\\
    \iprojr{V}{W}&\colon L(V\otimes G(W))\to L(V)\otimes W, &h\otimes_K (v\otimes w) &\mapsto (h_{(1)}\otimes_K v)\otimes h_{(2)}w,\\
    (\iprojr{V}{W})^{-1}&\colon L(V)\otimes W\to L(V\otimes G(W)), &(h\otimes_K v)\otimes w &\mapsto h_{(1)}\otimes_K (v\otimes S(h_{(2)})w).
\end{align*}
\end{lemma}

\subsection{Ambiadjunctions from Hopf algebra morphisms}

The following result derived from \cite{MN}*{Theorem~3.15} characterizes when induction and coinduction are isomorphic for a morphism of rings $\varphi\colon K\to H$ and, in particular, for a morphism of Hopf algebras.

\begin{proposition}[Menini--Nastasescu]\label{thm:ind-coind}
The following are equivalent for a morphism of rings $\varphi\colon K\to H$.
\begin{itemize}
    \item[(i)] The functors $\Ind_\varphi$ and  $\CoInd_\varphi$ are isomorphic.
    \item[(ii)] 
    \begin{itemize}
        \item[(a)] $H$ is finitely generated projective as a left $K$-module, and 
        \item[(b)] there is an isomorphism of $H$-$K$-bimodules
        $$\theta_K\colon H\to \Hom_{K}(H,K).$$
    \end{itemize}
    \item[(iii)]
       \begin{itemize}
        \item[(a)] $H$ is finitely generated projective as a left $K$-module, and 
        \item[(b)] there is a morphism of $K$-$K$-bimodules
        $$\tr \colon H \to K$$
which induces a bijection $\theta_K \colon H\to \Hom_K(H,K), h\mapsto \tr((-)h)$. Following \cite{FMS} we call $\tr$ a \emph{Frobenius morphism}.
    \end{itemize}
\end{itemize}
Under the equivalent conditions above, there are bijections between the following sets:
\begin{itemize}
    \item The set of natural isomorphisms $\theta\colon \Ind_\varphi\cong \CoInd_\varphi$;
    \item The set of isomorphisms of $H$-$K$-bimodules $\theta_K\colon H\to \Hom_K(H,K)$;
    \item The set of Frobenius morphisms $\tr\colon H\to K$.
\end{itemize}
\end{proposition}

\begin{remark}
The last condition (iii)(b) in \Cref{thm:ind-coind} is stated as the existence of a pairing $(-,-)\colon H\otimes_K H\to K$ in \cite{MN} which is left onto and non-degenerate. This is equivalent to the above description since, given a pairing $(-,-)$, we can define $\tr(h):=(1,h)$, and given a trace map, we set $(h,g):=\tr(hg)$. 

If $\varphi: K \rightarrow H$ is an inclusion of Hopf algebras over a field $\Bbbk$ and if $H$ is finite-dimensional over $\Bbbk$, then the condition that $H$ is finitely generated projective over $K$ is automatically satisfied, as $H$ is free over $K$ \cite{NZ}. 
\end{remark}

For later use, the next lemma describes how (iii) implies (i) by giving an explicit natural isomorphism.

\begin{lemma}\label{lem:thetaV}
If the conditions in \Cref{thm:ind-coind} hold, then a Frobenius morphism $\tr\colon H\to K$ 
gives a natural isomorphism $\theta\colon \Ind_\varphi\to \CoInd_\varphi$ with component maps
$$\theta_V(h\otimes_K v)=\big(g\mapsto \tr(gh)v \big).$$

for $V \in \lMod{K}$, $v \in V$, $h \in H$.
\end{lemma}

\begin{lemma}\label{lem:all-Frob-mors}
If the conditions in \Cref{thm:ind-coind} hold, then, for any Frobenius morphism $\tr\colon H\to K$, we have a bijection
\begin{align*}
\{ c \in H^{\times} \mid \forall x \in K: cx = xc \} &\rightarrow \{ \phi\colon H\to K \mid \text{$\phi$ is a Frobenius morphism} \}\\
c &\mapsto (h \mapsto \tr(hc))
\end{align*}
where $H^{\times} \subseteq H$ denotes the subset of units.
\end{lemma}
\begin{proof}
This follows from \Cref{thm:ind-coind} and the fact that the following map is a bijection:
\begin{align*}
\Aut( \Ind_\varphi ) &\rightarrow \{ c \in H^{\times} \mid \forall x \in K: cx = xc \} \\
\alpha &\mapsto \alpha_{K}(1) \qedhere
\end{align*}

\end{proof}

As a first example, we consider groups.

\begin{example}\label{ex:groups}
Let $\mathsf{G}$ be a group and $\sfK\subset \sfG$ a subgroup of finite index $n$. Consider the map 
$$\tr\colon \sfG\to K, \quad g\mapsto \begin{cases}
    g, & \text{if }g\in \sfK,\\
    0, &  \text{if }g\notin \sfK,
\end{cases}
$$
which extends to a morphism of $K$-$K$-bimodules $\tr\colon G=\Bbbk \sfG\to K=\Bbbk\sfK$. Note that choosing a coset decomposition $\sfG=\coprod_{i=1}^n \sfK g_i$ with $g_1=1$ shows that $G$ is free of finite rank as a left $K$-module and that $\Hom_K(G,K)$ has a basis given by the $K$-linear maps
$$\delta_i\colon G\to K, \quad g_i \mapsto \delta_{i,1}$$
where $\delta_{i,1}$ denotes the Kronecker delta.
Now, $\delta_1=\tr$ and under the $G$-action $g\cdot \tr(-)=\tr((-)g)$ we get that 
$g_i^{-1}\cdot \tr=\delta_i$. Hence, the equivalent conditions of \Cref{thm:ind-coind} hold for $\varphi$ being the inclusion $K\hookrightarrow G$.
\end{example}

The following counterexample shows that not even all embeddings of finite-dimensional Hopf algebras admit a Frobenius morphism.

\begin{example} \label{ex:Taft}
Let $\epsilon\in \Bbbk$ be a primitive $\ell$-th root of unity, with $\ell>1$, and denote $\mathsf{C}_\ell=\langle g|g^\ell\rangle$. Consider the \emph{Taft algebra} \cite{Taf}
$$T:=T_\ell(\epsilon)=\Bbbk \langle g,x\rangle/(g^\ell=1, x^\ell=0, gx=\epsilon xg),$$
which is a Hopf algebra with coproduct $\Delta$, counit $\varepsilon$, and antipode $S$  determined by 
$$\Delta(g)=g\otimes g, \quad \Delta(x)=x\otimes 1+g\otimes x, \quad \varepsilon(g)=1, \quad \varepsilon(x)=0, \quad S(g)=g^{-1}, \quad S(x)=-g^{-1}x.$$
%%Inductively, the coproduct and antipode are given by 
%%$$\Delta(x^k)=\sum_{l=0}^k \binom{k}{l}_\epsilon x^lg^{k-l}\otimes x^{k-l},\qquad S(x^k)=
%%where the $\epsilon$-binomial coefficients are
%%$$\binom{k}{l}_\epsilon=\frac{[k]_\epsilon!}{[l]_\epsilon![k-l]_\epsilon!},$$
%%for $[n]_\epsilon=1+\epsilon+\ldots +\epsilon^{n-1}$.
Note that $T$ is free as a left $K$-module, for $K=\Bbbk \sfC_\ell$, over the set $x^i$, $i=0,\ldots, \ell-1$. Now suppose that 
$\tr\colon T\to K$ is a $K$-bimodule map. Then $\tr$ is determined by the values $k_i=\tr(x^i)$. In particular, it follows that 
$$gk_i=\tr(gx^i)=\epsilon^i\tr(x^ig)=\epsilon^i k_ig=\epsilon^i gk_i.$$
This implies that $(1-\epsilon^i)k_i = 0$ and, thus, $k_i=0$ unless $i=0 \mod \ell$. In particular, $\tr( (-)x ) = 0$, which means that $\tr$ cannot provide an isomorphism as required in \Cref{thm:ind-coind}(ii)(b). Thus, $\Ind_\varphi$ and $\CoInd_\varphi$ are not isomorphic for $\varphi\colon K\hookrightarrow T$.
\end{example}

We now collect a few useful formulas as consequences of the above \Cref{lem:thetaV}, cf.~\cite{FMS}*{1.3.~Proposition}.

\begin{lemma}\label{lem:kideltai}
Suppose the conditions in \Cref{thm:ind-coind} hold. Then there exist generators $h_1,\ldots,h_n$ of $H$ as a left $K$-module and $\delta_1,\ldots, \delta_n$ of $H$ as a right $K$-module such that 
\begin{gather}
\sum_{i=1}^n \delta_i\otimes_K h_i h = \sum_{i=1}^n h\delta_i\otimes_K h_i, \qquad \forall h\in H, \label{eq:kideltai1} \\
h =\sum_{i=1}^n \varphi(\tr(h\delta_i))h_i= \sum_{i=1}^n \delta_i\varphi(\tr(h_i h)), \qquad \forall h\in H.
\label{eq:kideltai2}
\end{gather}
\end{lemma}
\begin{proof}
    By the dual basis lemma (see, e.g., \cite{Coh}*{Section 4.5, Proposition 5.5}) we can find elements $h_1,\ldots, h_n$ in $H$ and $f_1,\ldots, f_n$ in $H^\vee= \Hom_K(H,K)$ such that 
    \begin{equation}\label{eq:dualbasis}
    h=\sum_{i=1}^n \varphi(f_i(h))h_i , \qquad 
    f=\sum_{i=1}^n f_i f(h_i) \qquad\forall h\in H ,f \in H^\vee.
    \end{equation}
Recall the $H$-$K$-bimodule isomorphism $\theta_K(h)=\big(g\mapsto \tr(gh) \big)$ from \Cref{lem:thetaV} (where we set $V = K$) which satisfies $\theta_K(1)=\tr$.
We set 
$$\delta_i:=\theta_K^{-1}(f_i).$$
Then $f_i(h)=\tr(h\delta_i)$. Now \Cref{eq:kideltai2} follows from \Cref{eq:dualbasis}, for the second equality of \eqref{eq:kideltai2} by application of $\theta_K^{-1}$ to the second equality in \eqref{eq:dualbasis}.  Indeed,
$$ h =\theta_K^{-1}(\theta_K (h))=\sum_{i=1}^n \theta^{-1}_K (f_i)\varphi(\theta_K(h)(h_i))=\sum_{i=1}^n \delta_i\varphi(\tr(h_ih)).$$

To prove \Cref{eq:kideltai1}, first observe that left $H$-linearity of $\theta_K^{-1}$ and \Cref{eq:dualbasis} imply that
$$h \theta_K^{-1}(f)=\sum_{i=1}^n \delta_i\varphi(f(h_i h))
\qquad \forall h\in H, f\in H^\vee.
$$
Hence, again using the first equation in \Cref{eq:dualbasis}, we compute
\begin{align*}\sum_{i=1}^n \delta_i\otimes_K h_ih =& \sum_{i=1}^n \delta_i\otimes_K \sum_{j=1}^n \varphi(f_j(h_ih))h_j\\
=&\sum_{i,j=1}^n \delta_i\varphi(f_j(h_ih)) \otimes_K  h_j\\
=&\sum_{j=1}^n h\theta^{-1}_K(f_j)\otimes_K h_j=\sum_{i=1}^n h \delta_i \otimes_K h_i
\end{align*}
proving \Cref{eq:kideltai1}.
\end{proof}

With the notation of \Cref{lem:kideltai}, the inverse of $\theta_V$ from \Cref{lem:thetaV} can be described explicitly by
\begin{equation}
\theta_V^{-1}\colon \CoInd_\varphi(V)\to \Ind_\varphi(V), \quad f\mapsto \sum_{i=1}^n\delta_i\otimes_K f(h_i).\label{eq:thetainv}
\end{equation}
Thus, one derives the following ambiadjunction $\Ind_\varphi\dashv \Res_\varphi\dashv \Ind_\varphi$.
\begin{lemma}
Given that the equivalent conditions from \Cref{thm:ind-coind} hold, the functor $F=\Ind_\varphi$ is left and right adjoint to the monoidal functor $G=\Res_\varphi\colon \lMod{H}\to \lMod{K}$. The unit and counit morphisms can be described explicitly as
\begin{align}
    \unit^{G\dashv F}_W&\colon W\to FG(W), &w&\mapsto \sum_{i=1}^r \delta_i\otimes_K h_iw,\label{eq:unit1}\\ 
    \counit^{G\dashv F}_W&\colon GF(V)\to V, & h\otimes_K v& \mapsto \tr(h)v,\label{eq:counit1}\\
    \unit^{F\dashv G}_V&\colon V\to GF(V), &v&\mapsto 1\otimes_K v,\label{eq:unit2}\\
    \counit^{F\dashv G}_W&\colon FG(W)\to W, &h\otimes_K w&\mapsto hw,\label{eq:counit2}
\end{align}
for any  $K$-module $V$ and any $H$-module $W$.
\end{lemma}

Moreover, we can derive formulas for the lax and oplax structure from \eqref{eq:lax-Hopf}--\eqref{eq:oplax-Hopf}.

\begin{lemma}\label{lem:Hopf-laxoplax}
Given that the equivalent conditions from \Cref{thm:ind-coind} hold, the functor $F=\Ind_\varphi$ is lax and oplax monoidal with
\begin{gather*}
    \lax_{V,W}\colon  F(V)\otimes F(U)\to F(V\otimes U),\\ \quad (h\otimes_K v)\otimes (g\otimes_K u)\mapsto \sum_i \delta_i\otimes_K (\tr((h_i)_{(1)}h)v\otimes \tr((h_i)_{(2)}g)u),\\
\oplax_{V,W}\colon  F(V\otimes U)\to F(V)\otimes F(U),\\ \quad h\otimes_K (u\otimes v)\mapsto (h_{(1)}\otimes_K v)\otimes (h_{(2)}\otimes_K u),
\end{gather*}
together with the unit and counit morphisms
\begin{gather*}
   \lax_0\colon \one \to F(\one), \quad 1\mapsto \sum_{i=1}^r \delta_i\otimes_K \varepsilon_H(h_i), \label{eq:unit-Hopf}\\
    \oplax_0\colon F(\one)\to \one, \quad h\otimes_K1 \mapsto \varepsilon_H(h), \label{eq:counit-Hopf}
\end{gather*}
for the counit $\varepsilon_H$ of $H$.
\end{lemma}

We know from \cite{FLP2}*{Lemmas~7.1 and 7.10} that if $H$ is finitely generated projective over $K$, then the projection formula holds for $F$ as a left and right adjoint. Explicit formulas are given in the following lemma.

\begin{lemma}\label{lem:proj-Hopf}
Assuming that the equivalent conditions from \Cref{thm:ind-coind} hold, the projection formula isomorphisms are given by
\begin{align*}
    \projl{W}{V}^{F, G\dashv F}\colon W\otimes F(V)&\to F(G(W)\otimes V),  \\w\otimes (h\otimes_K v)&\mapsto \sum_{i=1}^r \delta_i\otimes_K ((h_i)_{(1)}w\otimes \tr((h_i)_{(2)}h)v),\\
    \projr{V}{W}^{F, G\dashv F}\colon F(V)\otimes W&\to F(V\otimes G(W)),  \\(h\otimes_K v)\otimes w&\mapsto \sum_{i=1}^r \delta_i\otimes_K (\tr((h_i)_{(1)}h)v\otimes (h_i)_{(2)}w),\\
    \projl{W}{V}^{F, F\dashv G}\colon F(G(W)\otimes V)&\to W\otimes F(V), \\h\otimes_K (w\otimes v)&\mapsto h_{(1)}w\otimes (h_{(2)}\otimes_K v),\\
    \projr{V}{W}^{F, F\dashv G}\colon F(V\otimes G(W))&\to F(V)\otimes W, \\h\otimes_K(v \otimes w) &\mapsto (h_{(1)}\otimes_K v)\otimes h_{(2)}w.
\end{align*}
\end{lemma}
\begin{proof}
The formulas are derived, 
using \Cref{lem:kideltai} and \Cref{eq:thetainv}, from \cite{FLP2}*{Section~7}.
\end{proof}

\subsection{Conditions for induction to be a Frobenius monoidal functor}
We now determine conditions for induction along a morphism of Hopf algebras $\varphi\colon K\to H$ to be a Frobenius monoidal functor. We require the following notation.

\begin{definition}[$H^\reg$, $K^\varphi$]\label{def:comodulestructure}
We denote by
$H^\reg$ the \emph{regular} right $H$-comodule with coaction
$$\delta_H^\reg=\Delta_H \colon H\to H\otimes H,$$
and by $K^\varphi$ the right $H$-comodule structure on $K$ given by  co-restricting $\delta_K^\reg$ along $\varphi$, i.e.,
$$\delta^\varphi=(\id\otimes \varphi)\Delta_K \colon K\to K\otimes H.$$
\end{definition}

\begin{proposition}\label{prop:Hopf-right}
Assume that $H$ is finitely generated projective as a left $K$-module and given a Frobenius morphism $\tr\colon H\to K $, see \Cref{thm:ind-coind}(iii). Then the following conditions are equivalent.
\begin{enumerate}[(i)]
    \item For $F=\Ind_\varphi$ and $G=\Res_\varphi$, the natural transformations $\projrnoarg^{F, F\dashv G}$ and $\projrnoarg^{F, G\dashv F}$ are mutual inverses.
    \item  The following elements of $H\otimes_K(K\otimes H)$ are equal:

    \begin{align}  \label{eq:cond-Hopf-Frobenius2}
        \sum_{i=1}^n \delta_i\otimes_K (\tr((h_{i})_{(1)})\otimes (h_{i})_{(2)})=1\otimes_K (1\otimes 1).
    \end{align}
    \item 
    The Frobenius morphism $\tr\colon H^\reg\to K^\varphi$ is a morphism of right $H$-comodules. 
\end{enumerate}
 Moreover, if a Frobenius morphism satisfying the equivalent conditions (i)--(iii) exists, then it is unique up to a non-zero scalar.
\end{proposition}
\begin{proof}
Assume that $H$ is finitely generated projective as a left $K$-module and assume given a Frobenius morphism $\tr\colon H\to K$. That is, the equivalent conditions from \Cref{thm:ind-coind} hold. Thus, $F\dashv G\dashv F$ is an ambiadjunction and the projection formula holds for both $F$ as a left and right adjoint by \Cref{lem:proj-Hopf}.

\emph{(i) $\Leftrightarrow$ (ii):}
The composition $\projr{W}{V}^{F,G\dashv F}\projr{W}{V}^{F,F\dashv G}$ is given by
\begin{align*}
\projr{W}{V}^{F,G\dashv F}\projr{W}{V}^{F,F\dashv G}(h\otimes_K (w\otimes v))=&
\sum_i h\delta_i\otimes_K (\tr((h_i)_{(1)})v\otimes (h_i)_{(2)}w).
\end{align*}
Whether this composition is equal to $\id_{F(W\otimes G(V))}$ can, equivalently, be verified on $V=K$ and $W=H$ and the generators $h=w=1_H, v=1_K$ which yields the equivalent condition \Cref{eq:cond-Hopf-Frobenius2}. 
As $H$ is finitely generated projective as a left $K$-module, by \Cref{lem:proj-Hopf}, both $\projrnoarg^{F,F\dashv G}$ and $\projrnoarg^{F, G\dashv F}$ are invertible. Thus, it suffices to check one composition is the identity.

\emph{(ii) $\Leftrightarrow$ (iii):}

Under the isomorphisms $K\otimes_K H\cong H\cong H\otimes_K K$, \Cref{eq:kideltai2} gives
\begin{equation}1\otimes_K h=\sum_{i=1}^n \tr(h\delta_i)\otimes_K h_i
\in K\otimes_K H,
%%\qquad h\otimes_K 1=\sum_{i=1}^n \delta_i\otimes_K \tr(h_ih)
 %%\in H\otimes_K K.
 \label{eq:H-selfdual2}
 \end{equation}

Hence, if \Cref{eq:cond-Hopf-Frobenius2} holds, we compute for any $h\in H$, 
\begin{align*}
1\otimes_K (\tr(h)_{(1)}\otimes \varphi(\tr(h)_{(2)}))&=
   \tr(h)\otimes_K (1\otimes 1)\\
   &= \sum_{i=1}^n \tr(h\delta_i)\otimes_K(\tr((h_i)_{(1)})\otimes (h_i)_{(2)})\\
    &= 1\otimes_K (\tr(h_{(1)})\otimes h_{(2)}).
\end{align*}
Here, the first equality uses the relative tensor product condition
$$k\otimes_K (l\otimes h)=1\otimes_K(k_{(1)}l\otimes \varphi(k_{(2)})h), \qquad \forall k,l\in K, h\in H,$$
in $K\otimes_K (K\otimes H)$, the second equality uses Equations \eqref{eq:cond-Hopf-Frobenius2} after application of $\tr$ to the left tensor factor, and the third uses \Cref{eq:H-selfdual2} after application of $\id\otimes_K(\tr\otimes\id)\Delta_H$. Applying the isomorphism $K\otimes_K (K\otimes H)\cong K\otimes H$ now implies that $\tr\colon H^\reg\to K^\varphi$ is a morphism of right $H$-comodules as claimed.

Conversely, assume that $\tr\colon H^\reg\to K^\varphi$ is a morphism of right $H$-comodules.  That is, 
$$\tr(h_{(1)})\otimes h_{(2)}=\tr(h)_{(1)}\otimes \varphi(\tr(h)_{(2)}).$$
Hence, 
\begin{align*}
    \sum_{i=1}^n \delta_i\otimes_K (\tr((h_i)_{(1)})\otimes (h_i)_{(2)})
    &=\sum_{i=1}^n \delta_i\otimes_K (\tr(h_i)_{(1)}\otimes\varphi(\tr(h_i)_{(2)}))\\
    &=\sum_{i=1}^n \delta_i\otimes_K (\tr(h_i)_{(1)}\otimes\tr(h_i)_{(2)}\cdot 1)\\
    &=\sum_{i=1}^n \delta_i\varphi(\tr(h_i))\otimes_K (1\otimes 1)=1\otimes_K( 1\otimes 1 ),
\end{align*}
where the last equality uses \Cref{eq:kideltai2}. Thus, \Cref{eq:cond-Hopf-Frobenius2} holds.

Last, if $\tr$ satisfies the equivalent conditions (i)--(iii), then $F \dashv G \dashv F$ can be regarded as an ambiadjunction of right $\cC$-modules (for $\cC = \lMod{H}$) by \Cref{theorem:strong_monoidal_to_ambiadjunction_right}.
By \Cref{remark:ambiadjunctions_classified_by_aut}, such ambiadjunctions are classified by the automorphisms of $F$ regarded as a right module functor. By \Cref{proposition:EndF}, we have
$\End(F) \cong \End_{\lMod{K}}(\one) \cong \Bbbk$. It follows that $\tr$ is uniquely determined up to a non-zero scalar.
\end{proof}

By \Cref{prop:Hopf-right}, we obtain the following result as a direct consequence of \Cref{theorem:frobenius_functors_frobenius_monoidal}.

\begin{corollary}\label{cor:Ind-Frob}
    Assume that $\varphi\colon K\to H$ is a morphism of Hopf algebras such that $H$ is finitely generated projective as a left $K$-module and that there exists a Frobenius morphism $\tr\colon H\to K$ which is a morphism of right $H$-comodules $\tr\colon H^\reg\to K^\varphi$. Then 
     $\Ind_\varphi\colon \lMod{K}\to \lMod{H}$, with lax and oplax monoidal structure from \Cref{lem:Hopf-laxoplax}, is a Frobenius monoidal functor.
\end{corollary}

We observe that, since $H$ is finitely generated over $K$ by assumption, $\Ind_\varphi$ restricts to the subcategories of finite-dimensional modules.

\begin{remark}
    We note that the Frobenius monoidal structure on $\Ind_\varphi$ obtained in \Cref{cor:Ind-Frob} changes when the Frobenius morphism $\tr$ is replaced by a non-zero scalar multiple $\lambda\tr$. Then we can replace $\delta_i$ by $\delta_i'=\lambda^{-1}\delta_i$ and, together with the same $h_i$, obtain a dual basis. Replacing $\tr$ by $\lambda\tr$, we observe that the oplax structure remains unchanged while $\lax$ is replaced by $\lambda\lax$ and $\lax_0$ is replaced by $\lambda^{-1}\lax_0$. This implies that the resulting Frobenius monoidal functors are not isomorphic. 
\end{remark}

\begin{example}\label{ex:uqsl2}
Consider the small quantum group $H=u_\epsilon(\fr{sl}_2)$ for $\epsilon\in \Bbbk=\mC$ a root of unity of odd order $\ell$, see e.g., \cite{Kas}*{Definition~VI.5.6}. We denote the generators 
of $u_\epsilon(\fr{sl}_2)$ by $e,f,k^{\pm 1}$ and recall that $k^{\pm}$ generate a group algebra $K=\Bbbk \sfC_\ell$ which is a Hopf subalgebra of $H=u_\epsilon(\fr{sl}_2)$. We let $\varphi$ denote the inclusion $K\hookrightarrow H$. 

The set $\{f^ie^j~|~0\leq i,j<\ell \}$ gives a basis for $H$ as a left (and right) $K$-module which we use to verify that a choice of Frobenius morphism is given by 
$$\tr\colon H\to K, \qquad f^ie^j\mapsto \delta_{i,\ell-1}\delta_{j,\ell-1} k^{\ell-1}.$$
Indeed, using the commutator relations, found e.g.\ in \cite{Kas}*{Lemma~VI.1.3}, one checks that 
$$e^if^jk^{1-\ell} \cdot \tr(f^me^n)=\tr(f^me^ne^if^jk^{1-\ell})=\delta_{i+n,\ell-1}\delta_{j+n,\ell-1}.$$
This implies that any dual basis element in $\Hom_K(H,K)$ can be generated by acting with elements from $H$ on $\tr$. In the sense of \Cref{lem:kideltai}, we have the following dual bases
$$\{h_{(i,j)}:= f^ie^j~|~0\leq i,j\leq \ell-1\}, \qquad  \{\delta_{(i,j)}:= e^{\ell-1-j}f^{\ell-1-i}k^{1-\ell}~|~0\leq i,j\leq \ell-1\}.$$

We now check condition (iii) in \Cref{prop:Hopf-right}:
\begin{align*}
    (\tr\otimes \id)\Delta_u(e^{\ell-1}f^{\ell-1})&=\epsilon^0\binom{\ell-1}{0}_\epsilon\binom{\ell-1}{\ell-1}_\epsilon \tr(e^{\ell-1}f^{\ell-1})\otimes e^0f^0k^{\ell-1}\\
    %%i=\ell-1, j=\ell-1, r=0 s=\ell-1
    &=k^{\ell-1}\otimes k^{\ell-1}\\
    &=(\id\otimes \varphi)\Delta_K(k^{\ell-1})=(\id\otimes \varphi)\Delta_K\tr(f^{\ell-1}e^{\ell-1}),
\end{align*}
where $(-)_\epsilon$ denotes the quantum binomial coefficient.
Here, we use \cite{Kas}*{Proposition~VII.1.3} for the coproduct $\Delta_H$.
Evaluating at $e^if^j$, with $i\neq \ell-1$ or $j\neq \ell-1$, instead of $e^{\ell-1}f^{\ell-1}$, we get zero on both sides of the equation as the highest order monomial $f^{\ell-1}e^{\ell-1}$ cannot appear in the tensor factors of the coproduct, and all other terms are annihilated by $\tr$. Hence, by \Cref{cor:Ind-Frob} we find that 
$$\Ind_\varphi\colon \lMod{\Bbbk \sfC_\ell}\to \lMod{u_\epsilon(\fr{sl}_2)}$$
has the structure of a Frobenius monoidal functor.
\end{example}

\begin{remark}\label{rem:uqsl2}
    \Cref{lem:all-Frob-mors} enables us to construct more Frobenius morphisms. For instance, in the above \Cref{ex:uqsl2}, there exists a family of Frobenius morphisms
    $$\tr_a\colon H\to K, \qquad f^ie^j\mapsto \delta_{i,\ell-1}\delta_{j,\ell-1}k^{\ell-1}a,$$
for $a\in K^\times$, a unit in the group algebra $K=\Bbbk \sfC_\ell$.

By \Cref{prop:Hopf-right} or a direct computation,  $\tr_a$ is a morphism of right $H$-comodules if and only if $a\in\Bbbk^\times$. This shows that whether $\projrnoarg^{F,F\dashv G}$ and $\projrnoarg^{F,G\dashv F}$ are mutual inverses depends, in general, on the choice of natural isomorphism $L\cong R$ between the left and right adjoint of the strong monoidal functor $G$. Hence, the equation of whether the lax and oplax monoidal structures on an ambiadjoint $F$ give a Frobenius monoidal functor depend on the choice of the unit and counit data of the ambiadjunction.
\end{remark}

\subsection{Frobenius monoidal functors on Yetter--Drinfeld module categories}

We now provide conditions  for the induction functor  along a morphism of Hopf  algebras $\varphi\colon K\to H$ to give a Frobenius monoidal functor $\cZ(\Ind_\varphi)\colon \cZ(\lMod{K})\to \cZ(\lMod{H})$ on Drinfeld centers. Extending \Cref{def:comodulestructure} we fix the following notation.

\begin{definition}[${}^\reg H^\reg$, ${}^\varphi K^\varphi$]\label{def:bicomodulestructure}
We denote by
${}^\reg H^\reg$ the $H$-$H$-bicomodule with left and right $H$-coactions given by the regular ones 
and by ${}^\varphi K^\varphi$ the $H$-$H$-bicomodule structure on $K$ given by $${}^\varphi \delta^\varphi=(\varphi\otimes \id\otimes \varphi)({}^\reg\delta_K^\reg) \colon K\to H\otimes K\otimes H,$$
i.e, given by co-restricting ${}^\reg\delta_K^\reg$ along $\varphi.$ 
\end{definition}

Recall that the Drinfeld center $\cZ(\lMod{H})$ is equivalent to the braided monoidal category $\lYD{H}$ of  \emph{Yetter--Drinfeld modules} (or \emph{YD modules}) over $H$ \cites{Yet,Maj,EGNO}.
This category $\lYD{H}$ has as objects $\Bbbk$-vector spaces $V$ which are both left $H$-modules and $H$-comodules with coaction
$$\delta\colon V\to H\otimes V, \qquad v\mapsto v^{(-1)}\otimes v^{(0)},$$
such that the YD condition
\begin{equation}\label{eq:YD-cond}
h_{(1)}w^{(-1)}\otimes h_{(2)}w^{(0)}=(h_{(1)}w)^{(-1)}h_{(2)}\otimes (h_{(1)}w)^{(0)}
\end{equation}
holds. Using the antipode $S$, this condition is equivalent to 
\begin{equation}\label{eq:YD-cond2}
\delta(h w)=h_{(1)}w^{(-1)}S(h_{(3)})\otimes h_{(2)}w^{(0)}.
\end{equation}
The half-braiding associated with a YD module $V$ is defined by
\begin{equation}\label{eq:YDbraiding}
    c^V_X\colon V\otimes W\to W\otimes V, \qquad v\otimes w \mapsto (v^{(-1)}\cdot w)\otimes v^{(0)}, 
\end{equation}
for any object $W$ of $\lMod{H}$ and $v\in V,w\in W$. In the following, we will work directly with the category of YD modules.

\smallskip

We can now give equivalent characterizations of the property that the projection formula morphisms for $\Ind_\varphi \dashv \Res_\varphi$ and $\Res_\varphi \dashv \Ind_\varphi$ are mutually inverse.

\begin{proposition}
\label{thm:Hopf}
Assume that $H$ is finitely generated projective as a left $K$-module and assume given a Frobenius morphism $\tr\colon H\to K $, see \Cref{thm:ind-coind}(iii). Then the following conditions are equivalent for $F=\Ind_\varphi$, $G=\Res_\varphi$.
\begin{enumerate}[(i)]
    \item $\projlnoarg^{F, F\dashv G}$ and $\projlnoarg^{F, G\dashv F}$ as well as  $\projrnoarg^{F, F\dashv G}$ and $\projrnoarg^{F, G\dashv F}$, are mutually inverse pairs of natural isomorphisms.
    \item  The equation
    \begin{align}\label{eq:cond-Hopf-Frobenius}
        \sum_{i=1}^n \delta_i\otimes_K ((h_i)_{(1)}\otimes \tr((h_i)_{(2)}))=1\otimes_K (1\otimes 1)\,\in\, H\otimes_K (H\otimes K),
    \end{align} and \Cref{eq:cond-Hopf-Frobenius2} hold.
    \item 
    The Frobenius morphism defines a morphism of $H$-$H$-bicomodules 
    $$\tr\colon {}^\reg H^\reg\to {}^\varphi K^\varphi.$$
\end{enumerate}
\end{proposition}
\begin{proof}
We have already shown the equivalence of the conditions corresponding to $\projrnoarg^{F, F\dashv G}$ and $\projrnoarg^{F, G\dashv F}$ being inverse natural isomorphisms in \Cref{prop:Hopf-right}. Equivalence of the conditions for the left projection formula morphisms is shown analogously. 
\end{proof}

Assume that
$H$ is finitely generated projective as a left $K$-module. For $V$ a Yetter--Drinfeld module over $K$ with coaction
$$\delta^V\colon V\to K\otimes V,\quad  \delta^V(v)=v^{(-1)}\otimes v^{(0)},$$
we define
\begin{align}\label{eq:YDcoactionInd}
    \delta^{\Ind_\varphi(V)}(h\otimes_K v)=h_{(1)}\varphi(v^{(-1)})S(h_{(3)})\otimes (h_{(2)}\otimes_K v^{(0)}).
\end{align}
We recall from \cite{FLP2}*{Corollary~7.2}
that the assignments 
\begin{gather}
    (V,\delta^V)\mapsto (\Ind_\varphi(V),\delta^{\Ind_\varphi(V)}), \qquad f\mapsto \Ind_{\varphi}(f),
\end{gather}
for any morphism $f$ in $\lYD{K}$, give a braided oplax monoidal functor $\cZ(\Ind_\varphi)\colon \lYD{K}\to \lYD{H}$ with oplax monoidal structure is given by \Cref{eq:oplax-Hopf}.

The following result is now a direct consequence of \Cref{thm:Hopf} and \Cref{thm:ZF}.

\begin{corollary}\label{cor:ZInd-Frob}
Assume that $H$ is finitely generated projective as a left $K$-module and assume given a Frobenius morphism $\tr\colon H\to K $. 
If $\tr\colon {}^\reg H^\reg\to {}^\varphi K^\varphi$ is a morphism of $H$-$H$-bicomodules, then $\cZ(\Ind_\varphi)\colon \lYD{K}\to \lYD{H}$ is a braided Frobenius monoidal functor with lax and oplax structure given by \Cref{lem:Hopf-laxoplax}.
\end{corollary}

\begin{remark}
    If $\tr\colon H\to K$ is only a morphism of right or left $H$-comodules, then the functor $\Ind_\varphi\colon \lMod{K}\to \lMod{H}$ is still Frobenius monoidal. This follows from \Cref{theorem:frobenius_functors_frobenius_monoidal} if $\tr$ is a morphism of right $H$-comodules and by a $\otimes^\oop$-dual statement, derived from \Cref{theorem:frobenius_functors_frobenius_monoidal}, if $\tr$ is a morphism of left $H$-comodules.

    In both of these situations, $\Ind_\varphi$ extends to Drinfeld centers in two different ways, once using $\Ind_\varphi$ as right or as left adjoint by \cite{FLP2}*{Theorems 4.10 and 4.15}. One of the functors is lax monoidal while the other one is oplax monoidal and we do not obtain a Frobenius monoidal functor on the Drinfeld centers unless both conditions hold.

    We also know in such a situation, there cannot be another Frobenius morphism $\tr'$ which is a morphism of $H$-$H$-bicomodules. This follows as by \Cref{prop:Hopf-right}, $\tr'$ is a non-zero scalar multiple of $\tr$ and hence $\tr'$ is a morphism of $H$-$H$-bicomodules if and only if $\tr$ is.
\end{remark}

\begin{example}\label{ex:groups2}
    As a first example, continuing \Cref{ex:groups}, we see that the equivalent conditions from \Cref{thm:Hopf} hold for group algebras $K\subseteq G$ such that $n=|\sfG:\sfK|<\infty$. Indeed, 
    given a choice of coset representatives, $\sfG=\coprod_{i=1}^n \sfK g_i$ we find that $\{h_i:=g_i\}_{1\le i\le n}$, $\{\delta_i:=g_i^{-1}\}_{1\le i\le n}$ are dual bases in the sense of \Cref{lem:kideltai}. Now we compute 
    \begin{align*}
        \sum_{i=1}^n g_i^{-1}\otimes g_i\otimes \tr(g_i)=g_1^{-1}\otimes g_1\otimes \tr(g_1)=1\otimes 1\otimes 1,
    \end{align*}
    since $\tr(g_i)=\delta_{i,1}$ and $g_1=1$ by choice. This shows that \Cref{eq:cond-Hopf-Frobenius} holds. By cocommutativity, \Cref{eq:cond-Hopf-Frobenius2} follows the same way. Thus, the equivalent conditions from \Cref{thm:Hopf} hold. The induced functor on Drinfeld centers is hence braided Frobenius monoidal, recovering the result from \cite{FHL}*{Proposition~B.1}.
\end{example}

\begin{example}\label{ex:classicalH}
    Consider the forgetful functor $G\colon \lMod{H}\to \Vect$ for a finite-dimensional Hopf algebra $H$. This functor corresponds to $\Res_{\varphi}$ for the inclusion $\varphi\colon \Bbbk\hookrightarrow H$. In this case, a Frobenius morphism $\tr\colon H\to \Bbbk$ is a cyclic generator for $H^*$ as a left $H$-module. In fact, any right integral for $H^*$, i.e., a linear map $\lambda\colon H\to \Bbbk$ satisfying 
    $$\lambda(h_{(1)})h_{(2)}=\lambda(h)1_H,$$
    for all $h\in H$ (cf.\ \cite{Rad}*{Section~10.1} gives such a Frobenius morphism $\tr=\lambda\colon H\to \Bbbk$ which, moreover, is a morphism of right $H$-comodules. Then, the conclusion of \Cref{cor:Ind-Frob} applies giving that $\Ind_\varphi(\Bbbk)$ is a Frobenius algebra which can be identified with $H$ as a coalgebra. The last statement in \Cref{prop:Hopf-right} recovers the fact that the space or right integrals for $H^*$ is one-dimensional. These statements recover the classical result of \cite{LS2}. 

Moreover, for $\tr=\lambda$ to be a morphism of $H$-$H$-bicomodules, $\lambda$ also needs to be a \emph{left} integral for $H$. This is the case when $H^*$ is \emph{unimodular} (see e.g.\ \cite{Rad}*{Definition~10.2.3}). In this case, \Cref{cor:Ind-Frob} applies and $H$ is a commutative Frobenius algebra in $\lYD{H}$ in this case. This leads to the following consequence.
\end{example}

\begin{corollary}\label{cor:Hunimodular}
If $H^*$ is a unimodular Hopf algebra, then $\Ind_\Bbbk(\one)$ has the structure of a commutative Frobenius algebra in $\lYD{H}$ which is isomorphic to $H$ as a coalgebra.
\end{corollary}
\begin{proof}
This result is a consequence of \Cref{cor:Ind-Frob} as explained in \Cref{ex:classicalH}.
The Yetter--Drinfeld module structure is given by the regular action and the adjoint coaction
$$h\cdot g=hg, \qquad \delta(h)=h_{(1)}S(h_{(3)})\otimes h_{(2)},$$
for $h,g\in H$. The coaction formula is derived from \Cref{eq:YDcoactionInd}.
The product $m$ on $H$ used here is obtained by dualizing the coproduct of $H$ along the non-degenerate pairing 
$$\inner{-,-}\colon H\otimes H\to \Bbbk, \quad \inner{h,g}=\lambda(hg).$$
That is, there exists elements $\delta_i,h_i\in H$ such that 
$$h=\sum_i\delta_i\lambda(h_ih), \qquad h=\sum_i h_i\lambda(h\delta_i),$$
for all $h\in H$ and the product is obtained from the lax structure of \Cref{lem:Hopf-laxoplax} as
$$m(h\otimes g)= \sum_i \delta_i \lambda((h_{i})_{(1)}h)\lambda((h_{i})_{(2)}g).$$
The coproduct obtained from the oplax structure is that of the Hopf algebra $H$. 
\end{proof}

We conclude with a detailed example of a morphism of Hopf algebras $\varphi$ such that, even though $\Ind_\varphi\cong \CoInd_\varphi$, only the right versions of the equivalent conditions of \Cref{thm:Hopf} are satisfied, but the left versions are \emph{not} satisfied. Thus, the Frobenius monoidal induction functor does not extend to centers in this example.

\begin{example}\label{ex:uqsl2-cont}
 Continuing \Cref{ex:uqsl2}, with $\varphi:K=\Bbbk \sfC_\ell \hookrightarrow u_\epsilon(\mathfrak{sl}_2)=H$, we note that $\tr$ is \emph{not} a morphism of \emph{left} $H$-comodules. This amounts to the computation that 
\begin{align*}
    (\id\otimes \tr)\Delta_H(f^{\ell-1}e^{\ell-1})&=e^0 f^0 k^{-\ell+1}\otimes \tr(e^{\ell-1} f^{\ell-1})\\
    %%i=\ell-1, j=\ell-1, r=0 s=\ell-1
    &=k^{-\ell+1}\otimes k^{\ell-1}\\
    &\neq (\id\otimes \varphi)\Delta(k^{\ell-1}),
\end{align*}
since $\ell> 2$. Hence, $\tr$ is \emph{not} a morphism of $H$-$H$-bicomodules. 
Therefore, by \Cref{thm:Hopf}, the ambiadjunction $\Ind_\varphi\dashv \Res_\varphi\dashv \Ind_\varphi$ satisfies $\projrnoarg^{F,F\dashv G}=(\projrnoarg^{F,G\dashv F})^{-1}$  but  $\projlnoarg^{F,F\dashv G}\neq (\projlnoarg^{F,G\dashv F})^{-1}$. In particular, the sufficient condition for the induced functors on the Drinfeld centers via \Cref{thm:ZF} to be Frobenius monoidal, does not hold.

Moreover, as any Frobenius morphism that is a morphism of right $H$-comodules is a scalar multiple of $\tr$, it is not possible to find a Frobenius morphism that is a morphism of \emph{both} left and right $H$-comodules.

Note that the functor $\Ind_\varphi$ still extends to Drinfeld centers as both a lax and an oplax monoidal functor by \cite{FLP2}*{Corollaries~7.2 and 7.11}, but with two distinct constructions of half-braidings. 
\end{example}

\begin{remark}\label{rem:uqsl2-cont}
We continue \Cref{rem:uqsl2} and \Cref{ex:uqsl2-cont} assuming that $\ell>2$ and consider the trace $\tr=\tr_{1}$. 
By the converse of \Cref{thm:Hopf}, we see that 
$$\projrnoarg^{F,F\dashv G}=(\projrnoarg^{F,G\dashv F})^{-1}\quad  \text{but} \quad\projlnoarg^{F,F\dashv G}\neq(\projlnoarg^{F,G\dashv F})^{-1},$$ where $K$ is the regular left $K$-module, $H$ is the regular left $H$-module. We can regard $K$ as an object in $\lYD{K}$ with the regular left action, given by multiplication, and trivial coaction, using that $K$ is cocommutative. 

Recall that we obtain two induced functors on YD modules \cite{FLP2}*{Corollaries~7.2 and 7.11}, namely
$$\cZ(\Ind_\varphi)\colon \lYD{K}\to \lYD{H}, \quad \cZ(\CoInd_\varphi)\colon \lYD{K}\to \lYD{H}.$$
Fixing an isomorphism $\Ind_\varphi\cong \CoInd_\varphi$ by virtue of choosing the Frobenius morphism $\tr_1$ from \Cref{rem:uqsl2}, both functors can be modelled on $\Ind_\varphi$. In particular, 
$$ \cZ(\Ind_\varphi)(K)=(H,\delta^1), \qquad \cZ(\CoInd_\varphi)(K)=(H,\delta^2),$$
where $\delta^1, \delta^2\colon H\to K\otimes H$ are the two coactions
\begin{align*}
    \delta^1(h)&=h_{(1)}S(h_{(3)})\otimes h_{(2)},\\
    \delta^2(h)&=\sum_i (\delta_i)_{(1)}(h_{i})_{(2)}\otimes (\delta_i)_{(2)}\tr((h_i)_{(1)}h).
\end{align*}
We can compare the images of the generator $1$ for the cyclic module $H=\Ind_\varphi(K)$. This yields
$$\delta^1(1)=1\otimes 1, \qquad \delta^2(1)=k^{-2}\otimes 1.$$
As $\ell>2$, these two coactions are not equal. In particular, by the Yetter--Drinfeld condition for $\delta^2$ we find that 
\begin{align*}
    \delta^2(h)&=h_{(1)} k^{-2} S(h_{(3)})\otimes h_{(2)}.
\end{align*}
We note that the YD modules $(H,\delta^1)$ and $(H,\delta^1)$ cannot be isomorphic. To be isomorphic as YD modules, we need an isomorphism of left $H$-modules. In particular, the image of $1$ needs to be a cyclic generator.

Note that the object $\CoInd_\varphi(K)\cong H$ has a commutative algebra structure in $\lYD{H}$ that is not the given algebra structure of $H$. For instance, its unit map is given by $\CoInd_\varphi(\Bbbk \xrightarrow{1\mapsto 1_K} K)\lax_0$, which corresponds to the element
$$\sum_{0\leq a,b\leq \ell-1} e^{\ell-1-b} f^{\ell-1-a} k^{-1}\otimes_K \epsilon_H(f^a e^b)1_K= e^{\ell-1}f^{\ell-1}k^{-1}\in H,$$
using the formula from \Cref{eq:lax-Hopf}.
\end{remark}

\bibliography{biblio}
\bibliographystyle{amsrefs}

\end{document}